\documentclass[11pt]{amsart}

\usepackage{amsmath}
\usepackage{amssymb}
\usepackage{amsthm}
\usepackage[all]{xy}
\usepackage{hyperref}
\usepackage[dvipsnames]{xcolor}

\usepackage{enumitem}
\usepackage{fullpage}

\numberwithin{equation}{section}
\newtheorem{theorem}{\bf Theorem}[section]

\newtheorem{corollary}[theorem]{\bf Corollary}
\newtheorem{lemma}[theorem]{\bf Lemma}
\newtheorem{proposition}[theorem]{\bf Proposition}

\theoremstyle{definition}

\newtheorem{definition}[theorem]{\bf Definition}

\theoremstyle{remark}
\newtheorem{remark}[theorem]{Remark}
\newtheorem{example}[theorem]{Example}

\numberwithin{equation}{section}
\makeatletter
\let\c@theorem\c@equation
\makeatother

\newtheorem*{namedtheorem}{\theoremname}
\newcommand{\theoremname}{testing}

\renewcommand{\leq}{\leqslant}
\renewcommand{\geq}{\geqslant}

\newcommand{\sseq}{\subseteq}

\newcommand{\gen}[1]{\langle #1 \rangle}

\renewcommand{\hat}{\widehat}

\newcommand{\conj}{\operatorname{conj}}

\newcommand{\xra}{\xrightarrow}

\newcommand{\bD}{\mathbf{D}}
\newcommand{\bW}{\mathbf{W}}

\newcommand{\C}{\mathcal{C}}

\newcommand{\F}{\mathcal{F}}

\renewcommand{\L}{\mathcal{L}}
\newcommand{\M}{\mathcal{M}}
\renewcommand{\O}{\mathcal{O}}

\newcommand{\T}{\mathcal{T}}

\newcommand{\W}{\mathcal{W}}

\newcommand{\Z}{\mathcal{Z}}
\renewcommand{\phi}{\varphi}

\newcommand{\typ}{\textup{typ}}

\DeclareMathAlphabet\EuR{U}{eur}{m}{n}
\SetMathAlphabet\EuR{bold}{U}{eur}{b}{n}
\newcommand{\curs}{\EuR}
\newcommand{\Ab}{\curs{Ab}}

\newcommand{\id}{\operatorname{id}}

\newcommand{\Hom}{\operatorname{Hom}}
\newcommand{\Mor}{\operatorname{Mor}}

\newcommand{\Iso}{\operatorname{Iso}}
\newcommand{\Ob}{\operatorname{Ob}}

\newcommand{\Aut}{\operatorname{Aut}}
\newcommand{\Out}{\operatorname{Out}}
\newcommand{\Inn}{\operatorname{Inn}}

\newcommand{\op}{\operatorname{op}}

\begin{document}
\title[Rigid automorphisms of linking systems]{Rigid automorphisms of linking systems}
\author{George Glauberman}
\address{Department of Mathematics \\ University of Chicago \\ 5734 S.
University Ave \\ Chicago, IL 60637}
\email{gg@math.uchicago.edu}
\author{Justin Lynd}
\address{Department of Mathematics \\ University of Louisiana at Lafayette \\ Lafayette, LA 70504}
\email{lynd@louisiana.edu}
\date{\today}
\thanks{The first author is partially supported by a Simons Foundation
Collaboration Grant. The second author was partially supported by NSA Young
Investigator Grant H98230-14-1-0312, Marie Curie Fellowship No. 707778, and NSF
Grant DMS-1902152 during the preparation of this manuscript. The authors thank
these organizations for their support. The first author expresses his gratitude
to the Institute of Mathematics at the University of Aberdeen for its
hospitality during a research visit in 2017. The authors would also like to
thank the Isaac Newton Institute for Mathematical Sciences, Cambridge, for
support and hospitality during the programme Groups, representations and
applications, where work on this paper was undertaken and supported by EPSRC
grant no EP/R014604/1.}

\begin{abstract}
A rigid automorphism of a linking system is an automorphism which restricts to
the identity on the Sylow subgroup. A rigid inner automorphism is conjugation
by an element in the center of the Sylow subgroup. At odd primes, it is known
that each rigid automorphism of a centric linking system is inner.  We prove
that the group of rigid outer automorphisms of a linking system at the prime
$2$ is elementary abelian, and that it splits over the subgroup of rigid inner
automorphisms. In a second result, we show that if an automorphism of a finite
group $G$ restricts to the identity on the centric linking system for $G$, then
it is of $p'$-order modulo the group of inner automorphisms, provided $G$ has
no nontrivial normal $p'$-subgroups.  We present two applications of this last
result, one to tame fusion systems.
\end{abstract}

\maketitle

\section{Introduction}
A saturated fusion system $\F$ is a category in which the objects are the
subgroups of a fixed finite $p$-group $S$, and the morphisms are injective
group homomorphisms between subgroups which are subject to axioms first
outlined by Puig \cite{Puig2006, AschbacherKessarOliver2011}. When $G$ is a
finite group with Sylow $p$-subgroup $S$, there is a saturated fusion system
$\F = \F_S(G)$ in which the morphisms are the $G$-conjugation maps between
subgroups.  One of the important properties of this category is that it keeps
precisely the data required to recover the homotopy type of the Bousfield-Kan
$p$-completion $BG_p^\wedge$ of the classifying space of $G$, as shown in the
Martino-Priddy Conjecture, proved by Oliver \cite{Oliver2004,Oliver2006}.
Recovery of $BG_p^\wedge$, or a $p$-complete space denoted $B\F$ when no group
$G$ is associated with $\F$, is based on the construction of a centric
linking system $\L$ for $\F$, an extension category of $\F$ whose existence and
uniqueness up to rigid isomorphism was first established in general by
Chermak \cite{Chermak2013}.  From a group theoretic point of view, centric
linking systems, or more generally the transporter systems of Oliver-Ventura
\cite{OliverVentura2007} and the localities of Chermak \cite{Chermak2013},
provide finer approximations to $p$-local structure. They abstract the
transporter categories of finite groups, and form structures appearing in new
recent approaches to revising the classification of finite simple groups. 

We study here in more detail the comparison maps between automorphism groups of
finite groups, linking systems, and fusion systems. When $\L$ is a centric
linking system associated to the fusion system $\F$, there are groups of
automorphisms $\Aut(\L)$ and $\Aut(\F)$, and a map $\tilde{\mu}\colon \Aut(\L)
\to \Aut(\F)$ given essentially by restriction to the Sylow group $S$.  When
$\L = \L_S^c(G)$ and $\F = \F_S(G)$ for some finite group $G$, there is also a
comparison map $\tilde{\kappa}_G\colon N_{\Aut(G)}(S) \to \Aut(\L)$, where
$N_{\Aut(G)}(S)$ consists of those automorphism of $G$ which leave $S$
invariant.  These induce a pair of maps
\[
\Out(G) \xra{\kappa_G} \Out(\L) \xra{\mu_\L} \Out(\F)
\]
on outer automorphism groups. We write $\Aut_0(\L)$ for the group of
\emph{rigid automorphisms} of $\L$, namely $\ker(\tilde{\mu}_\L)$.  Similarly,
$\Out_0(\L)$ is short for $\ker(\mu_\L)$. 

It follows from the exact sequence of
\cite[III.5.12]{AschbacherKessarOliver2011} and Chermak's Theorem that $\mu_\L$
is an isomorphism if $p$ is odd, and is surjective with kernel an abelian
$2$-group when $p = 2$.  Moreover, the surjectivity of $\kappa_G$ has been
studied intensively in articles by Andersen, Oliver, and Ventura, and by
Broto, Moller, and Oliver.  

Our first result extends the consequences of unique existence of centric
linking systems to show that the kernel of $\mu_\L$ is in fact of exponent at
most $2$, in general, when $p = 2$. To make it easier to apply, we state and
prove this in the slightly more general setting of a linking locality
(defined just below), and in three equivalent ways.  Set $k(p) = 1$ if $p$ is
odd, and $k(p) = 2$ if $p = 2$. In particular, a group of exponent $k(p)$ is
the trivial group if $p$ is odd and is elementary abelian if $p = 2$. 

\begin{theorem}[Linking locality version]
\label{T:main-loc}
If $(\L, \Delta, S)$ is a linking locality at the prime $p$, then the group
$\Out_0(\L)$ of rigid outer automorphisms of $\L$ is abelian of exponent at
most $k(p)$. Moreover, the exact sequence
\[
1 \to \Aut_{Z(S)}(\L) \to \Aut_0(\L) \to \Out_0(\L) \to 1
\]
splits. 
\end{theorem}

\begin{theorem}[Linking system version]
\label{T:main-trans}
If $\L$ is a linking system at the prime $p$ (in the general sense of
\cite{Henke2019}), then the group $\Out_0(\L)$ of rigid outer automorphisms of
$\L$ is abelian of exponent at most $k(p)$. Moreover, the exact sequence
\[
1 \to \Aut_{Z(S)}(\L) \to \Aut_0(\L) \to \Out_0(\L) \to 1
\]
splits.
\end{theorem}

\begin{theorem}[Cohomological version]
\label{T:main-cohom}
Let $\F$ be a saturated fusion system over the finite $p$-group $S$, let $\O(\F^c)$
be the orbit category of $\F$-centric subgroups, and let $\Z_\F\colon
\O(\F^{c})^{\op}\to \mathsf{Ab}$ denote the center functor.  Then
$\lim^1\Z_\F$ is of exponent at most $k(p)$.  Moreover, the exact sequence
\[
1 \to \hat{B}(\O(\F^c), \Z_\F) \to \widehat{Z}^1(\O(\F^c), \Z_\F) \to {\lim}^1 \Z_\F \to 1 
\]
splits.
\end{theorem}

Here, a \emph{linking locality} in the sense of \cite{Henke2019} (also called a
\emph{proper locality} in \cite{ChermakFL2}), is a locality $(\L,\Delta,S)$
such that $\Delta$ contains all subgroups of $S$ which are centric and radical
in $\F = \F_S(\L)$, the fusion system of $\L$, and such that
$C_{N_\L(P)}(O_p(N_\L(P))) \leq O_p(N_\L(P))$ for each $P \in \Delta$.
Similarly, a \emph{linking system} is a transporter system $\L$ associated with
a saturated fusion system $\F$ such that $\Ob(\L)$ contains all $\F$-centric
radical subgroups and such that $C_{\Aut_\L(P)}(O_p(\Aut_\L(P))) \leq
O_p(\Aut_\L(P))$ for each $P \in \Ob(\L)$. Other definitions of the term
``linking system'' without further qualification, such as in
\cite[Definition~III.4.1]{AschbacherKessarOliver2011}, are special cases of
this one. 

An automorphism of a locality $\L$ is \emph{inner} if it is induced by
conjugation by an element of $N_\L(S)$, and a similar remark applies to
transporter systems.  In the case of a linking locality or linking system, a
rigid inner automorphism is conjugation by an element of the center of $S$. We
have denoted the group of rigid inner automorphisms by $\Aut_{Z(S)}(\L)$. This
helps to explain some of the terminology and notation in
Theorems~\ref{T:main-loc}-\ref{T:main-trans}.  We explain in more detail in
Section~\ref{S:prelim}. Terminology used in Theorem~\ref{T:main-cohom} is
recalled in Section~\ref{S:rigid-cent}.

As mentioned above, when $p$ is odd and $\L$ is a centric linking system,
Theorems~\ref{T:main-loc}-\ref{T:main-cohom} follow from the proof of existence
and uniqueness of centric linking systems as given in \cite{Oliver2013} or
\cite{GlaubermanLynd2016}.  The connection between existence and uniqueness and
the higher limits of the center functor $\Z_\F$ over the orbit category
$\O(\F^c)$ of $\F$-centric subgroups is given by
\cite[Proposition~III.5.12]{AschbacherKessarOliver2011}.  In particular, this
result identifies $\Out_0(\L)$ with the first derived limit
${\lim}^1_{\O(\F^c)} \Z_\F$ of the center functor.  So when $p$ is odd the
theorems follow from \cite[Theorem~3.4]{Oliver2013} or
\cite[Theorem~1.1]{GlaubermanLynd2016} and an argument, provided in
Section~\ref{S:descent}, which uses Chermak's iterative procedure for extending
a given locality to a new locality on a larger object set.

We shall prove Theorem~\ref{T:main-loc} first in the case of a centric linking
locality, i.e., when $\Delta$ is the collection of $\F$-centric subgroups.  The
proof is applicable for all primes $p$, and so we obtain an alternative,
somewhat simpler proof of the triviality of $\Out_0(\L)$ for $p$ odd,
independent of the main result of \cite{GlaubermanLynd2016}.  We then deduce
Theorem~\ref{T:main-trans} in the same special case, along with
Theorem~\ref{T:main-cohom}. Afterward, we shall prove in
Section~\ref{S:descent} that this implies the seemingly more general statements
in Theorems~\ref{T:main-loc} and \ref{T:main-trans}.  

Along the way, we extend to transporter systems a result of Oliver on
isomorphisms of (quasicentric) linking systems (Proposition~\ref{P:aut-trans}),
and we interpret Chermak's work in the Appendix of \cite{Chermak2013} as an
equivalence of groupoids between localities and transporter systems
(Theorem~\ref{T:local-equiv-trans}). Besides their use in deducing
Theorem~\ref{T:main-trans} from \ref{T:main-loc}, one motivation for these
extensions is to make clear that the results of \cite{Oliver2013,
GlaubermanLynd2016} give existence and uniqueness of centric linking localities
up to \emph{rigid} isomorphism in the same way as the main theorem of
\cite{Chermak2013}. That this is not clear at first is caused by an ambiguity
in which the notion of ``isomorphism'' of transporter commonly in use does not
restrict to the notion of ``automorphism'' commonly in use, but rather to what
should be called ``rigid automorphism''. 

Automorphisms of a finite group that centralize a Sylow subgroup have been
studied by Glauberman, Gross, and others. The main result here can be seen as a
generalization to linking systems of \cite[Theorem~10]{Glauberman1968wc}. The
current work bears the same relationship to \cite[Theorem~10]{Glauberman1968wc}
as the proof of existence and uniqueness of centric linking systems outlined
above does to the work of Gross \cite{Gross1982} and to the recent work of the
authors with Guralnick and Navarro \cite{GGLN2019}.  Our proof of
Theorem~\ref{T:main-loc} is very different from the proof of
\cite[Theorem~10]{Glauberman1968wc}, however, in part because not all subgroups
of $S$ need be objects. 

Recall that for a finite group $G$ with Sylow $p$-subgroup $S$ and centric
linking system $\L_S^c(G)$, there is a comparison homomorphism $\kappa_G \colon
\Out(G) \to \Out(\L_S^c(G))$. It is induced essentially by restriction to
$p$-local structure modulo $p'$-cores, at the level of centric subgroups. In
the course of trying to recover from the above theorems the corresponding
results about finite groups, we were led to the following result, which seems
to be of independent interest. 

\begin{theorem}
\label{T:main2}
Let $p$ be a prime and $G$ a finite group with Sylow $p$-subgroup $S$. If
$O_{p'}(G) = 1$, then the kernel of the map $\kappa_G \colon \Out(G) \to
\Out(\L_S^c(G))$ is a $p'$-group.
\end{theorem}

The proof of Theorem~\ref{T:main2} relies on the $Z_p^*$-theorem, namely
the statement that an element $x \in S$ whose only $G$-conjugate in $S$ is $x$
itself must lie in the center of $G$ modulo $O_{p'}(G)$.  Thus, our proof of
Theorem~\ref{T:main2} relies on the Classification of Finite Simple Groups
(CFSG) if $p$ is odd. (This result and its corollaries in Section~\ref{S:grp}
for $p$ odd are the only results in the paper that depend on the CFSG.)  

When $G$ is simple, the cokernel of $\kappa_G$ has been studied extensively in
\cite{AOV2012}, \cite{BrotoMollerOliver2019}, and elsewhere. In particular, it
has now been shown that the fusion system of each finite simple group $G$ is
\emph{tame} in the sense of \cite{AOV2012}, namely, there is a possibly
different finite group $G'$ with Sylow subgroup $S$ such that $\F_S(G) \cong
\F_S(G')$ such that the map $\kappa_{G'}$ is split surjective.
Theorem~\ref{T:main2} has been shown in several special cases in the context of
those works, cf.  \cite[Lemma~5.9,Theorem~5.16]{BrotoMollerOliver2019}.

Theorem~\ref{T:main2} is proved as Theorem~\ref{T:kerkappa} in
Section~\ref{S:grp}, and we give two applications of it: we show that the
splitting condition in the definition of a tame fusion system may be
removed, and we give an interesting reinterpretation of the first author's
work on the Schreier conjecture \cite{Glauberman1966b}. 

\subsection*{Terminology and notation} 
When $G$ is a group and $g \in G$, we write $c_g$ for the left-handed
conjugation homomorphism $x \mapsto gxg^{-1}$ and its restrictions. The image
of a subgroup $P$ under $c_g$ is sometimes written in left-handed exponential
notation ${ }^gP$. We write $\Hom_G(P,Q)$ for the set $\{c_g \mid { }^gP \leq
Q\}$ of conjugation homomorphisms between $P$ and $Q$ induced in $G$.  Given a
finite group $G$ with Sylow $p$-subgroup $S$, the fusion system $\F_S(G)$ is
the category with objects the subgroups of $S$ and with morphism sets
$\Hom_{\F_S(G)}(P,Q) := \Hom_G(P,Q) := \{c_g \mid { }^gP \leq Q\}$.  Our
terminology for fusion systems follows \cite{AschbacherKessarOliver2011}. For
example, $\F^c$ denotes the set of $\F$-centric subgroups, $\F^{r}$ denotes the
set of $\F$-radical subgroups, $\F^f$ denotes the set of fully $\F$-normalized
subgroups, and concatenation in the superscript denotes the intersection of the
relevant sets. 

\section{Transporter systems and localities}\label{S:prelim} 
Throughout this section, $\F$ is a saturated fusion system over a $p$-group
$S$, and $\Delta$ is a nonempty collection of subgroups of $S$ which is closed
under $\F$-conjugacy and passing to overgroups. Fix also another triple $\F'$,
$S'$, and $\Delta'$ of this type. 

\subsection{Transporter systems}

In the case where $\F = \F_S(G)$ for some finite group $G$ with Sylow
$p$-subgroup $S$, the transporter category $\T_\Delta(G)$ of $G$ with object
set $\Delta$ is the category with morphisms $\Mor_{\T_{\Delta}(G)}(P,Q) =
N_G(P,Q) = \{g \in G \mid { }^gP \leq Q\}$ where composition is given
by multiplication in $G$. There is an inclusion functor $\delta \colon
\T_\Delta(S) \to \T_{\Delta}(G)$, as well as a functor $\pi\colon \T_\Delta(G)
\to \F_S(G)$ which is the inclusion on objects and which sends $g \in N_G(P,Q)$
to $c_g \in \Hom_G(P,Q)$, conjugation by $g$.  This is the standard example of
a transporter system associated with $\F_S(G)$. 

\begin{definition}[{\cite[Definition~3.1]{OliverVentura2007}}]\label{D:trans}
A \emph{transporter system} associated with $\F$ is a nonempty category
$\T$ with object set $\Delta \subseteq \Ob(\F)$, together with structural
functors 
\[
\T_\Delta(S) \xra{\delta} \T \xra{\pi} \F
\] 
which satisfy the following axioms.
\begin{enumerate}
\item[(A1)] $\Delta$ is closed under $\F$-conjugacy and upon passing to
overgroups, $\delta$ is the identity on objects, and $\pi$ is the inclusion on
objects.
\item[(A2)] For each $P, Q \in \Delta$, the kernel 
\[
E(P) := \ker(\pi_{P,P}\colon \Aut_\T(P) \to \Aut_\F(P))
\]
acts freely on $\Mor_{\T}(P,Q)$ by right composition, and $\pi_{P,Q}$ is the
orbit map for this action. In particular, $\pi_{P,Q}$ is surjective. Also,
$E(Q)$ acts freely on $\Mor_{\T}(P,Q)$ by left composition.  Here,
$\Aut_\T(P)$ denotes $\Mor_{\T}(P,P)$. 
\item[(B)] For each $P,Q \in \Delta$, $\delta_{P,Q}\colon N_S(P,Q) \to
\Mor_{\T}(P,Q)$ is injective, and the composite $\pi_{P,Q} \circ \delta_{P,Q}$
sends $g \in N_S(P,Q)$ to $c_g \in \Hom_{\F}(P,Q)$. 
\item[(C)] For each $\phi \in \Mor_{\T}(P,Q)$ and each $g \in P$, the
diagram
\[
\xymatrix{
P  \ar[r]^\phi & Q \\
P \ar[u]^{\delta_{P,P}(g)} \ar[r]_\phi & Q \ar[u]_{\delta_{Q,Q}(\pi(\phi)(g))}
}
\]
commutes in $\T$.
\item[(I)] $\delta_{S,S}(S)$ is a Sylow $p$-subgroup of $\Aut_\T(S)$. 
\item[(II)] Let $\phi \in \Iso_{\T}(P,Q)$, $P \triangleleft \bar{P} \leq S$,
and $Q \triangleleft \bar{Q} \leq S$ be such that $\phi \circ
\delta_{P,P}(\bar{P}) \circ \phi^{-1} \leq \delta_{Q,Q}(\bar{Q})$. Then there
is $\bar{\phi} \in \Mor_{\T}(\bar{P},\bar{Q})$ such that $\bar{\phi} \circ
\delta_{P,\bar{P}}(1) = \delta_{Q,\bar{Q}}(1) \circ \phi$. 
\end{enumerate}
\end{definition}

From now on,  we abbreviate $\delta_{P,P}$ to $\delta_{P}$, $\pi_{P,P}$ to
$\pi_{P}$, and use similar notation when considering the application of an
arbitrary functor on morphism sets. Also, any future reference to axioms
(A1)-(II) should be interpreted as reference to the axioms given in
Definition~\ref{D:trans}. The following lemma collects some basic properties
of morphisms in a transporter system. 

\begin{lemma}
\label{L:resext}
Fix a transporter system $(\T, \delta, \pi)$ associated with $\F$. 
\begin{enumerate}
\item[(a)] Each morphism in $\T$ is both a monomorphism and an epimorphism
in the categorical sense.
\item[(b)] (Restrictions are unique) Given objects $P_0 \leq P$, $Q_0 \leq Q$, and
two morphisms $\phi_0$, $\phi_0'$ making the diagram
\[
\xymatrix{
P  \ar[r]^{\phi} & Q \\
P_0 \ar[u]^{\delta_{P_0,P}(1)} \ar[r]_{\phi_0, \phi_0'} & Q_0 \ar[u]_{\delta_{Q_0,Q}(1)}
}
\]
commute, one has $\phi_0 = \phi_0'$. 
\item[(c)] (Extensions are unique) Given objects $P_0 \leq P$, $Q_0 \leq Q$, and two
morphisms $\phi$, $\phi'$ making the diagram
\[
\xymatrix{
P  \ar[r]^{\phi, \phi'} & Q \\
P_0 \ar[u]^{\delta_{P_0,P}(1)} \ar[r]_{\phi_0} & Q_0 \ar[u]_{\delta_{Q_0,Q}(1)}
}
\]
commute, one has $\phi = \phi'$. 
\end{enumerate}
\end{lemma}
\begin{proof}
Parts (a) and (b) are contained in \cite[Lemma~3.2]{OliverVentura2007}, while
part (c) is proved in \cite[Lemma~A.5(c)]{Chermak2013}.
\end{proof}

By a morphism of fusion systems $\F \to \F'$, it is meant a pair
$(\alpha,\Phi)$ where $\alpha \colon S \to S'$ is a group homomorphism and
$\Phi\colon\F \to \F'$ is a functor which together satisfy $\alpha(P) =
\Phi(P)$ on objects and $\Phi(\phi) \circ \alpha = \alpha \circ \phi$ for each
morphism $\phi$ in $\F$. If $\alpha$ is an isomorphism, then $\Phi$ is
determined uniquely by $\alpha$. So an isomorphism of fusion systems may be
regarded as an isomorphism of the underlying $p$-groups which ``preserves
fusion''.

\begin{definition}[Isomorphisms of transporter systems] 
\label{D:iso-trans}
Let $(\T, \delta, \pi)$ and $(\T', \delta', \pi')$ be transporter systems with
object sets $\Delta$ and $\Delta'$, for the saturated fusion systems $\F$ and
$\F'$, respectively.
\begin{enumerate}
\item Let $\alpha \colon \T \to \T'$ be an equivalence of categories. It
is said that
\begin{itemize}
\item $\alpha$ is \emph{isotypical} if $\alpha(\delta_P(P)) =
\delta'_{\alpha(P)}(\alpha(P))$ for each subgroup $P \in \Delta$, and that
\item $\alpha$ \emph{sends inclusions to inclusions} if
$\alpha(\delta_{P,Q}(1)) = \delta'_{\alpha(P), \alpha(Q)}(1)$ for each $P, Q
\in \Delta$. 
\end{itemize}
\item An \emph{isomorphism} is an equivalence $\T \to \T'$ which is isotypical
and sends inclusions to inclusions.  An automorphism is an isomorphism of
a transporter system onto itself. 
\item An isomorphism $\alpha \colon \T \to \T'$ is said to be \emph{rigid} if
$S = S'$ and $\alpha_S \circ \delta_S = \delta'_S$ as homomorphisms $S \to
\Aut_{\T'}(S)$. Here, as before, $\alpha_S$ means $\alpha_{S,S}$.
\item An automorphism $\alpha$ of $\T$ is \emph{inner} if there is an element
$\phi \in \Aut_\T(S)$ such that $\alpha$ is given on objects by $P \mapsto
c_\phi(P) := \pi(\phi)(P)$ and on morphisms by mapping $\psi\colon P \to Q$
to 
\[
c_{\phi}(\psi) := \phi|_{Q,c_\phi(Q)} \circ \psi \circ (\phi|_{P,c_{\phi}(P)})^{-1},
\]
where, for example, $\phi|_{Q,c_{\phi}(Q)}$ is the unique morphism from $Q$
to $c_{\phi}(Q)$ in $\T$ such that $\phi \circ \delta_{Q,S}(1) =
\delta_{c_{\phi}(Q), S}(1) \circ \phi$, as given by
Lemma~\ref{L:resext}(b). We refer to $c_{\phi}$ as \emph{conjugation by
$\phi$}. 
Write $\Aut_{Z(S)}(\T)$ for the group of rigid inner automorphisms of $\T$
which are conjugation by elements of $\delta_S(Z(S)) \leq \Aut_\T(S)$. 
\end{enumerate}
Denote by $\Aut(\T) := \Aut(\T,\delta,\pi)$ the group of automorphisms of $\T$.
Denote by $\mathsf{T}$ the category of transporter systems and isomorphisms. 
\end{definition}

\begin{remark}
An isomorphism of transporter systems is in particular an invertible functor,
and so one sees that $\Aut(\T)$ is indeed a group.  This was shown for linking
systems in \cite[Lemma~1.14(a)]{AOV2012}, and the same argument applies for an
arbitrary transporter system.  
\end{remark}

We have defined isomorphism here in analogy with the definition of an
automorphism of a centric linking system
\cite[III.4.3]{AschbacherKessarOliver2011}, but more generally than is usually
done.  The usual definition of an isomorphism of transporter systems is a
functor $\alpha\colon \T \to \T'$ which commutes with the structural functors:
$\alpha \circ \delta = \delta'$ and $\pi' \circ \alpha = \pi$. See for example
\cite[p.799]{BrotoLeviOliver2003}, \cite[Proposition~3.11]{OliverVentura2007},
\cite[p.146]{AschbacherKessarOliver2011}, or
\cite[Definition~A.2]{Chermak2013}.  Rather, Definition~\ref{D:iso-trans}
specializes to the definition of an automorphism of a linking system in
\cite[Section III.4.3]{AschbacherKessarOliver2011}.  

The following proposition extends Proposition~4.11 of
\cite{AschbacherKessarOliver2011} in two ways, but the proof follows the same
basic outline.  It helps explain that an isomorphism between transporter
systems is equivalent to a triple of functors commuting with the structural
functors, and that the usual definition of isomorphism of transporter systems
is the same as what we are calling a rigid isomorphism.

\begin{proposition}\label{P:aut-trans}
Fix transporter systems $(\T, \delta, \pi)$ and $(\T', \delta', \pi')$
associated to $\F$ and $\F'$ with object sets $\Delta$ and $\Delta'$ which
contain $\F^{cr}$ and ${\F'}^{cr}$. Given an isomorphism $\alpha \colon \T \to
\T'$ in the sense of Definition~\ref{D:iso-trans}, there is a unique associated
isomorphism $\beta \colon S \to S'$, a unique functor $\beta_* \colon
\T_\Delta(S) \to \T_{\Delta'}(S')$, and a unique isomorphism $c_\beta\colon
\colon \F \to \F'$ of fusion systems such that the diagram
\begin{eqnarray}
\label{E:diag-isom-trans}
\vcenter{
\xymatrix{
\T_\Delta(S) \ar[r]^\delta \ar[d]^{\beta_*} & \T \ar[r]^\pi \ar[d]^\alpha & \F \ar[d]^{c_{\beta}}\\
\T_{\Delta'}(S') \ar[r]^{\delta'}  & \T' \ar[r]^{\pi'} & \F'.
}
}
\end{eqnarray}
commutes and $\beta = (\beta_*)_{S}$. Moreover, $\alpha$ is a rigid
isomorphism if and only if both $\beta_*$ and $c_{\beta}$ are the identity
functors. 
\end{proposition}
\begin{proof}
Let $\alpha \colon \T \to \T'$ be an isomorphism. As $S$ is the only object of
$\T$ with the property that $\Mor_\T(P,S) \neq \varnothing$ for each object $P$
of $\T$, and the same is true for $S'$ with respect to $\T'$, it follows that
$\alpha(S) = S'$. So $\alpha_S(\delta_S(S)) = \delta'_{S'}(S')$ since $\alpha$
is isotypical.  By axiom (B) for a transporter system, the map $\delta'_{S'}
\colon S' \to \delta'_{S'}(S')$ is an isomorphism, so there is a unique map
$\beta$ from $S = \Aut_{\T_\Delta(S)}(S)$ to $S' = \Aut_{\T_{\Delta'}(S')}(S')$
such that
\begin{eqnarray} \label{E:betadef} \alpha_S(\delta_S(s)) =
\delta'_{S'}(\beta(s)) 
\end{eqnarray}
for each $s \in S$. Then $\beta = (\delta')_{S'}^{-1} \circ \alpha_S \circ
\delta_S$ is an isomorphism from $S$ to $S'$. Now $\alpha$ sends inclusions to
inclusions, so commutes with restrictions. Hence, for each $P \in
\Delta$, as $\alpha(\delta_P(P)) = \delta'_{\alpha(P)}(\alpha(P))$, we have
$\alpha_S(\delta_S(P)) = \delta'_{S'}(\alpha(P))$, and this shows with
\eqref{E:betadef} and injectivity of $\delta'$ that $\beta(P) = \alpha(P)$ for
each $P$. 

Let $\beta_* \colon \T_\Delta(S) \to \T_{\Delta'}(S')$ be the functor induced
by $\beta$. Namely, $\beta_*$ sends an object $P$ to $\beta(P)$, and it sends a
morphism $P \xra{s} Q$ to $\beta(P) \xra{\beta(s)} \beta(Q)$. Then $\delta'
\circ \beta_* = \alpha \circ \delta$ by construction. 

Next, we wish to define a functor $c_\beta\colon \F \to \F'$ via
a mapping on objects sending $P$ to $\beta(P)$, and on morphisms sending $P
\xra{\phi} Q$ to $\beta(P) \xra{\beta \circ \phi \circ \beta^{-1}} \beta(Q)$.
This is an isomorphism of fusion systems (the one corresponding to the
isomorphism $\beta$ from $S$ to $S'$) with inverse $c_{\beta^{-1}}$, if
well-defined.  In order to show the assignment is well-defined, we must prove
that each $\beta \circ \phi \circ \beta^{-1}$ is a morphism in $\F'$. This will
be done by showing that $c_\beta(\phi) = \pi'(\alpha(\tilde{\phi}))$ for each
$\tilde{\phi} \in \Mor_\T(P,Q)$ with $\pi(\tilde{\phi}) = \phi$, thus
simultaneously showing that the right square in \eqref{E:diag-isom-trans} commutes. 

Fix such a lift $\tilde{\phi}$ of $\phi$, and let $s \in P$. Consider the
following diagrams:
\[
\xymatrix{
P \ar[r]^{\tilde{\phi}} \ar[d]_{\delta_P(s)} & Q \ar[d]^{\delta_Q(\phi(s))\,,}\\
P \ar[r]_{\tilde{\phi}} & Q\\
}
\qquad 
\xymatrix{
\alpha(P) \ar[r]^{\alpha(\tilde{\phi})} \ar[d]_{\alpha(\delta_P(s))} & \alpha(Q) \ar[d]^{\alpha(\delta_Q(\phi(s)))\,,}\\
\alpha(P) \ar[r]_{\alpha(\tilde{\phi})} & \alpha(Q) \\
}
\qquad 
\xymatrix{
\beta(P) \ar[r]^{\alpha(\tilde{\phi})} \ar[d]_{\delta'_{\beta(P)}(\beta(s))} & \beta(Q) \ar[d]^{\delta'_{\beta(Q)}(\beta(\phi(s)))}\\
\beta(P) \ar[r]_{\alpha(\tilde{\phi})} & \beta(Q) \\
}
\]
By axiom (C) for $\T$, the first diagram commutes, and the second is
$\alpha$ applied to the first. As shown above, $\beta(P) = \alpha(P)$ and
$\alpha \circ \delta = \delta' \circ \beta_*$, so the third diagram is the same
as the second. By axiom (C) for $\T'$ with $\alpha(\tilde{\phi})$ and
$\beta(s)$ in the roles of $\phi$ and $g$, the morphism
$\delta'_{\beta(Q)}(\pi'(\alpha(\tilde{\phi}))(\beta(s)))$ in place of
$\delta'_{\beta(Q)}(\beta(\phi(s)))$ also makes the third diagram commute, so
we have
\[
\delta'_{\beta(Q)}(\beta(\phi(s))) \circ \alpha(\tilde{\phi})=
\delta'_{\beta(Q)}(\pi'(\alpha(\tilde{\phi}))(\beta(s))) \circ \alpha(\tilde{\phi}) 
\]
as morphisms between $\beta(P)$ and $\beta(Q)$ in $\T$.  Since each morphism in
a transporter system is an epimorphism (Lemma~\ref{L:resext}(a)) and
$\delta'_{\beta(Q)}$ is injective (axiom (B)), it follows that
\[
\beta(\phi(s)) = \pi'(\alpha(\tilde{\phi}))(\beta(s)), \quad \text{ for $s \in P$}.
\]
Hence, after replacing $s$ by $\beta^{-1}(s)$, we see that $c_\beta(\phi) =
\pi'(\alpha(\tilde{\phi}))$ as claimed, and this completes the proof of
existence of the functors $\beta_*$ and $c_{\beta}$. 

It remains to prove uniqueness. Observe that uniqueness of $\beta$ would follow
from that of $\beta_*$. Suppose $\gamma \colon \T_\Delta(S) \to
\T_{\Delta'}(S')$ is a functor such that $\gamma$ in place of $\beta^*$
makes the left square in \eqref{E:diag-isom-trans} commute.  Since $\delta$
and $\delta'$ are the identity on objects by axiom (A1), $\gamma$ agrees with
$\beta_*$ on objects.  Similarly they agree on morphisms, given commutativity
of the diagram, since ${\delta'}_{P,Q}$ is injective by axiom (B) for each $P,Q
\in \Delta$. Hence, $\gamma = \beta_*$.  Next, suppose in addition that $\eta
\colon \F \to \F'$ is another functor such that right square in
\eqref{E:diag-isom-trans} commutes with $\eta$ in place of $c_{\beta}$. By
axiom (A1), the functors $c_\beta$ and $\eta$ agree with $\alpha$ on the
objects $\Delta$. For each morphism $\phi$ in $\T$ between subgroups in
$\Delta$, we have $\eta(\pi(\phi)) = c_\beta(\pi(\phi))$, so by axiom (A2) on
the surjectivity of $\pi$ on morphism sets, we see that $\eta$ and $c_{\beta}$
agree on morphisms in $\F$ between subgroups in $\Delta$. By assumption
$\F^{cr} \sseq \Delta$, so the Alperin-Goldschmidt fusion theorem
\cite[Proposition~A.10]{BrotoLeviOliver2003} or
\cite[I.3.5]{AschbacherKessarOliver2011} gives equality. 

If $\alpha$ is a rigid isomorphism, then by definition $S = S'$. By
commutativity of the left square in \eqref{E:diag-isom-trans}, $\delta'_S
\circ \beta = \alpha_S \circ \delta_S = \delta'_S$. So $\beta = \id_S$ as
$\delta'_S$ is injective.  It was shown above that $\beta_*$ and $c_{\beta}$
are uniquely determined by $\beta$, so $\beta_*$ and $c_{\beta}$ are the
identity.  Conversely, if $\beta_*$ is the identity functor, then $S = S'$, and
by commutativity of the left square, we have $\alpha_S\circ \delta_S =
\delta'_S \circ \id_S = \delta'_S$, so $\alpha$ is rigid.
\end{proof}

As in the setting of (quasicentric) linking systems \cite[p.197]{AOV2012},
one can define a group homomorphism relating automorphisms of a transporter
system with automorphisms of the associated fusion system in this more general
setting, using Proposition~\ref{P:aut-trans}.  Let $(\T,\delta,\pi)$ be a
transporter system with object set $\Delta$ associated with the saturated
fusion system $\F$ on $S$. Assume that $\F^{cr} \subseteq \Delta$. Define 
\[
\tilde{\mu}_{\T}\colon \Aut(\T) \to \Aut(\F)
\]
to be the map which sends $\alpha \in \Aut(\T)$ to the automorphism
$\delta_S^{-1} \circ \alpha_S \circ \delta_S$ of $S = \Aut_{\T_\Delta(S)}(S)$.
Thus, $\tilde{\mu}_\T(\alpha)$ is the automorphism $\beta$ in
Proposition~\ref{P:aut-trans}. This is a group homomorphism (using uniqueness
of $c_\beta$) which maps $\Aut_\T(S)$ onto $\Aut_\F(S)$ and has
kernel $\Aut_0(\T)$. It induces a homomorphism
\[
\mu_\T \colon \Out(\T) \to \Out(\F)
\]
with kernel $\Out_0(\T)$.  When $\T = \T_\Delta(G)$ for some finite group $G$
with Sylow $p$-subgroup $S$, we sometimes write $\tilde{\mu}_G$ for
$\tilde{\mu}_\T$ and $\mu_G$ for $\mu_\T$, provided $\T$ is understood from
the context. 

\subsection{Localities}
In his proof of the existence and uniqueness of centric linking systems,
Chermak introduced localities and showed in \cite[Appendix]{Chermak2013} they
are essentially equivalent to transporter systems.  The purpose of this section
is to explain how Chermak's results give an equivalence of categories between
transporter systems and localities, with morphisms isomorphisms, while setting
up notation.

Let $\L$ be a finite set (we shall consider only finite
localities).  Write $\bW(\L)$ for the monoid of words $(f_n,\dots,f_1)$ in
the elements of $\L$, where the multiplication is concatenation $\circ$.  A
\emph{partial group} is a set $\L$ together with a subset $\bD := \bD(\L)
\subseteq \bW(\L)$, a multivariable product $\Pi\colon \bD \to \L$ defined on
words in $\bD$, and an inversion map $(-)^{-1}\colon \L \to \L$, subject to
certain axioms which may be found in \cite[Definition~2.1]{Chermak2013}. The
product $f_n \cdots f_1$ is \emph{defined} if $(f_n,\dots,f_1) \in \bD$, and in
this case we set $f_n \cdots f_1 = \Pi(f_n,\dots,f_1)$.  A partial group is a
group if and only if $\bD = \bW(\L)$, that is, all products are defined. A
\emph{partial subgroup} is a subset $\L_0$ of $\L$ with domain $\bD_0 \subseteq
\bW(\L_0) \cap \bD$, such that the restriction of the product $\Pi$ to $\bD_0$
is the product $\Pi_0$ for $\L_0$.  The subgroups of $\L$ are the partial
subgroups $\L_0$ with $\bW(\L_0) \subseteq \bD(\L)$. A \emph{homomorphism} of
partial groups is a function $\gamma \colon \L \to \M$ such that
$\gamma^*(\bD(\L)) \subseteq \bD(\M)$ and $\Pi(\gamma^*(w)) = \gamma(\Pi(w))$
for any word $w \in \bD(\L)$. Here, $\gamma^* \colon \bW(\L) \to \bW(\M)$ is
the map on words determined by $\gamma$. Partial groups and partial group
homomorphisms form a category, so there is the usual notion of isomorphism in
this category. A homomorphism $\gamma$ as above is an isomorphism if and only
if it is a bijective homomorphism satisfying $\gamma^*(\bD(\L)) = \bD(\M)$.

There is a natural notion of conjugation in a partial group when defined.
Given $f \in \L$, write $\bD(f)$ for the set of $x \in \L$ such that
$(f,x,f^{-1}) \in \bD$. The product $fxf^{-1} = \Pi(f,x,f^{-1})$ is the
conjugate of $x$ by $f$, sometimes written ${ }^fx$. A usual convention, which
we adopt, is that any such expression carries the tacit assumption that $x \in
\bD(f)$. Likewise, for any subset $X \subseteq \L$, the expression ${ }^f X$
has a similar meaning, including that $X \subseteq \bD(f)$. 

\begin{definition}\label{D:locality}
Let $\L$ be a finite partial group, let $S$ be a $p$-subgroup of $\L$, and let
$\Delta$ be a collection of subgroups of $S$. The triple $(\L, \Delta,S)$ is a
\emph{locality} if
\begin{enumerate}
\item[(L1a)] $\bD(\L)$ is equal to the set of those $(f_n,\dots,f_1) \in \bW(\L)$ such
that there is $(X_0,\dots,X_n) \in \bW(\Delta)$ with ${ }^{f_{i+1}}X_i =
X_{i+1}$ for each $0 \leq i < n$.
\item[(L1b)] If $P \in \Delta$ and $f \in \L$ with $P \leq \bD(f)$ and ${ }^fP \leq
S$, then $Q \in \Delta$ for each ${ }^fP \leq Q \leq S$. 
\item[(L2)] $S$ is a maximal member of the poset of $p$-subgroups of $\L$. 
\end{enumerate}
\end{definition}

We next set up some notation when working with a locality $(\L,\Delta,S)$.
A word $w = (f_n,\dots,f_1) \in \bW(\L)$ is \emph{in $\bD(\L)$ via $X_0$}
if ${}^{f_i \cdots f_1}X_{0} \in \Delta$ for each $1 \leq i \leq n$, compare
(L1a). For $f \in \L$, denote by $S_f$ the set of $s \in S$ such that ${ }^fs
\in S$. By \cite[Proposition~2.11]{Chermak2013}, $S_f \in \Delta$. In
particular, $S_f$ is a subgroup of $\L$ which plays the role of a Sylow
intersection.  For an object $P \in \Delta$, the normalizer $N_\L(P) = \{f \in
\L \mid { }^fP = P\}$,  and centralizer $C_\L(P) = \{f \in \L \mid { }^fx = x
\text{ for all $x \in P$}\}$ are subgroups of $\L$.

The \emph{fusion system} $\F_S(\L)$ of $\L$ is the fusion system on $S$ with
morphisms being those group monomorphisms between subgroups of $S$
which can be written as compositions of restrictions of the conjugation
homomorphisms $c_f \colon P \to Q$, $x \mapsto { }^fx$ between objects $P,Q \in
\Delta$. It is said that $\L$ is a locality \emph{on} $\F_S(\L)$. 

\begin{example}[{\cite[Example/Lemma~2.10]{Chermak2013}}]\label{Ex:locality}
Let $G$ be a finite group, let $S$ be a Sylow $p$-subgroup of $G$, and let
$\Delta$ be a collection of subgroups of $S$ which is closed under
$\F_S(G)$-conjugacy and upon passing to overgroups, and which contains
all $\F_S(G)$-centric radical subgroups.  Let $\L$ be the subset of $G$
consisting of those $g \in G$ such that there exists $P \in \Delta$ with ${
}^gP \leq S$ (so that ${ }^gP \in \Delta$). Let $\bD \subseteq \bW(\L)$ denote
the collection of all words $(g_n,\dots,g_1) \in \bW(\L)$ such that there is
$(X_0,\dots,X_n) \in \bW(\Delta)$ with ${ }^{g_i \cdots g_1}X_0 \in \Delta$ for
each $0 \leq i \leq n$.  Whenever $(g_n,\dots,g_1)$ is a word in $\bD$, define
$\Pi(g_n,\dots,g_1) = g_n\cdots g_1$, the product in $G$. Then
$(\L,\Delta,S)$ is a locality on $\F_S(G)$, written $\L_{\Delta}(G)$. 
\end{example}

\begin{definition}[Isomorphisms of localities]
Let $(\L, \Delta,S)$ and $(\L',\Delta',S')$ be localities. 
\begin{enumerate}
\item An isomorphism from $(\L,\Delta,S)$ to $(\L',\Delta',S')$ is an
isomorphism of partial groups $\beta\colon \L \to \L'$ such that $\beta(\Delta)
= \Delta'$ (hence, $\beta(S) = S'$). An automorphism of $(\L,\Delta,S)$ is an
isomorphism of $(\L,\Delta,S)$ to itself.
\item An isomorphism $\beta$ is \emph{rigid} if $S = S'$, and $\beta$ is the
identity on $S$. 
\item An automorphism $\alpha$ of $\L$ is \emph{inner} if it is given by
conjugation by an element of $N_\L(S)$, namely, there is $f \in N_\L(S)$ such
that $\alpha(x) = fxf^{-1}$ for all $x \in \L$. (Note that the product
$fxf^{-1}$ is always defined when $f \in N_\L(S)$.) 
\end{enumerate}
Write $\Aut(\L) := \Aut(\L,\Delta,S)$ for the group of automorphisms of $\L$,
$\Aut_0(\L)$ for the subgroup of rigid automorphisms, and $\Aut_{Z(S)}(\L)$ for
the subgroup of $\Aut_0(\L)$ consisting of automorphisms which are conjugation
by elements in $Z(S)$.  Denote by $\mathsf{L}$ the category of localities
with isomorphisms.
\end{definition}

\subsection{Equivalence between transporter systems and localities}

In \cite[Appendix]{Chermak2013}, Chermak goes most of the way toward proving
that there is an equivalence between the category of transporter systems with
rigid isomorphisms (in the sense of Definition~\ref{D:iso-trans}) and the
category of localities with rigid isomorphisms.  Here, we suggest a mild
extension of Chermak's results to an equivalence of the slightly larger
categories $\mathsf{T}$ and $\mathsf{L}$ with the same objects. First, we
briefly review how to pass from a locality to a transporter system and vice
versa. More details are given in \cite[Appendix A]{Chermak2013}.

\subsubsection{From localities to transporter systems}\label{sss:loc-to-trans}
Given a locality $(\L,\Delta,S)$, one can make a transporter system
$(\T_\Delta(\L),\delta,\pi)$ associated with $\F_S(\L)$ in the following
way. Let $\T_{\Delta}(\L)$ have object set $\Delta$, and for each $P,Q \in
\Delta$, take 
\[
\Mor_{\T_{\Delta}(\L)}(P,Q) = \{(f,P,Q) \mid { }^fP \leq Q\}.
\]
Composition is given by multiplication in $\L$. The functor $\delta$ is the
identity on objects, and sends $P \xra{s} Q$ to $(s,P,Q)$. The functor $\pi$ is
the inclusion on objects and sends $(f,P,Q)$ to the conjugation homomorphism
$c_f \colon P \to Q$. 

\subsubsection{From transporter systems to localities}\label{sss:trans-to-loc}
Conversely, to make a locality given a transporter system $(\T,\delta,\pi)$,
consider the collection of isomorphisms $\Iso(\T)$ in $\T$ and the following
relation on the set $\Mor(\T)$ of morphisms in $\T$: the morphism $\phi\colon P
\to Q$ is an extension of $\phi_0 \colon P_0 \to Q_0$, written $\phi_0 \uparrow
\phi$, if the diagram
\[
\xymatrix{
P \ar[r]^{\phi} & Q\\
P_0 \ar[u]^{\delta_{P_0,P}(1)} \ar[r]_{\phi_0} & Q_0 \ar[u]_{\delta_{Q_0,Q}(1)}
}
\]
commutes in $\T$. This is a partial order, and the equivalence relation on
$\Iso(\T)$ generated by its restriction to $\Iso(\T)$ is denoted $\equiv$.  It
is shown in \cite[Lemma~A.8(a)]{Chermak2013} that each $\equiv$-class has a
unique maximal member with respect to $\uparrow$. Write $[\phi]$ for the
equivalence class of $\phi$, and set $(\L,\Delta,S) = (\Iso(\T)/\!\!\equiv,
\Delta, S)$, where by abuse of notation, $S$ is identified with the set of
equivalence classes $\{[\delta_S(s)] \mid s \in S\}$ of elements in
$\delta_S(S) \subseteq \Aut_\T(S) \subseteq \Iso(\T)$. The domain
$\bD(\L_{\Delta}(\T))$ for the product is the set of all words $(f_n,\dots,f_1)
\in \bW(\L_{\Delta}(\T))$ such that there exist objects $P_0,\dots,P_n \in
\Delta$ and isomorphisms $\phi_i \colon P_{i-1}\to P_i$ in $\T$ such that
$\phi_i \in f_i$ for each $i$. In this situation, the product $\Pi \colon
\bD(\L_{\Delta}(\T)) \to \L_{\Delta}(\T)$ is defined by $\Pi(f_n,\dots,f_1) =
[\phi_n \circ \cdots \circ \phi_1]$.

Recall that $\mathsf{T}$ denotes the category of transporter systems with
isomorphisms and $\mathsf{L}$ denotes the category of localities with
isomorphisms. We write $\mathsf{T}_0$ and $\mathsf{L}_0$ for the categories of
transporter systems and localities with rigid isomorphisms.

\begin{theorem}[{cf. Chermak \cite[Appendix]{Chermak2013}}]
\label{T:local-equiv-trans}
The categories $\mathsf{T}$ and $\mathsf{L}$ are equivalent via a functor which
restricts to an equivalence between $\mathsf{T}_0$ and
$\mathsf{L}_0$. 
\end{theorem}
\begin{remark}
Strictly speaking, in order for the restriction of the functor $\mathsf{T}
\to \mathsf{L}$ (to be constructed in the proof) to induce an equivalence
between $\mathsf{T}_0$ and $\mathsf{L}_0$, we must make two canonical
identifications of $S$ with other incarnations of $S$. It is possible that a
more precise statement could be made involving a category of $S$-rigid
localities, where an $S$-rigid locality is a locality $\L$ together with an
embedding $S \hookrightarrow \L$ of partial groups which satisfies natural
conditions. But we do not pursue that, since our interest here is mainly in
Corollary~\ref{C:equiv-aut}.
\end{remark}
\begin{proof}[Proof of Theorem~\ref{T:local-equiv-trans}]
Define functors $\Theta \colon \mathsf{L} \to \mathsf{T}$ and $\Lambda
\colon \mathsf{T} \to \mathsf{L}$ as follows.  On objects, the functors are
as described in Subsections~\ref{sss:loc-to-trans} and
\ref{sss:trans-to-loc}.  Let $\gamma \colon \L \to \L'$ be an isomorphism
between the two localities $(\L,\Delta,S)$ and $(\L',\Delta',S')$. Define a
functor $\Theta(\gamma) \colon \T_{\Delta}(\L) \to \T_{\Delta'}(\L')$ by the
rule
\begin{align*} 
P &\mapsto \gamma(P),\\
(f,P,Q) &\mapsto (\gamma(f), \gamma(P), \gamma(Q)). 
\end{align*}
$\Theta(\gamma)$ is an invertible functor with inverse $\Theta(\gamma^{-1})$, it
is clearly isotypical, it sends inclusions to inclusions because $\gamma(1) =
1$, and hence it is an isomorphism of transporter systems. Observe that if
$\Delta = \Delta'$ (so $S = S'$) and $\gamma$ is a rigid isomorphism, then
$\Theta(\gamma)(\delta_S(s)) = (s,S,S) = \delta'_S(s)$ for each $s \in S$, so
$\Theta(\gamma)$ is a rigid isomorphism of transporter systems.  It is then
clear that $\Theta$ determines a functor $\mathsf{L} \to \mathsf{T}$, which
restricts to send $\mathsf{L}_0 \to \mathsf{T}_0$. 

Conversely, given an isomorphism $\alpha \colon \T \to \T'$, form the
associated localities $(\L_\Delta(\T), \Delta, S)$ and $(\L_{\Delta'}(\T'),
\Delta', S')$ and define a function $\Lambda(\alpha) \colon \L_\Delta(\T) \to
\L_{\Delta'}(\T')$ via $\Lambda(\alpha)([\phi]) = [\alpha(\phi)]^{'}$, where
here we write $[-]^{'}$ for equivalence classes in $\Iso(\T')$.  As $\alpha$ is
invertible, it induces a bijection $\Delta \to \Delta'$ sending $S \mapsto S'$
and a bijection $\Iso(\T) \to \Iso(\T')$.  Since $\alpha$ sends inclusions
to inclusions, it preserves $\uparrow$ and $\equiv$, and hence
$\Lambda(\alpha)$ is a well-defined bijection. Given that $\alpha$ is a
functor, it follows from the definition of multiplication in $\L_\Delta(\T)$
and \cite[Lemma~A.7(b)]{Chermak2013} that $\Lambda(\alpha)$ is a partial group
homomorphism. Then $\Lambda(\alpha)$ restricts to a homomorphism from
$S$ to $S'$ (if we identify these with $\{[\delta_S(s)] \mid s \in S\}$ and
$\{[\delta'_{S'}(s')] \mid s' \in S'\}$ via $\delta$ and $\delta'$,
respectively), because $\alpha$ is isotypical. Further, if $\alpha$ is rigid,
then this translates directly to the condition that $\Lambda(\alpha)$ is a
rigid isomorphism of localities.  Again, $\Lambda(\alpha^{-1})$ is the inverse
of $\Lambda(\alpha)$, and so $\Lambda(\alpha)$ is an isomorphism of localities.
Thus $\Lambda$ is a functor which restricts to send $\mathsf{T}_0 \to
\mathsf{L}_0$. 

Define $\eta \colon \id_{\mathsf{T}} \to \Theta \circ \Lambda$ as follows.
For any transporter system $\T$, $\eta_{\T} \colon \T \to \Theta(\Lambda(\T))$
sends each object to itself, and it sends a morphism $\phi \colon P \to Q$ in
$\T$ to the triple $([\phi_0], P, Q)$, where $\phi_0$ is the unique morphism
from $P$ to $Q_0 := \pi(\phi)(P)$ in $\T$ such that $\delta_{Q_0,Q}(1) \circ
\phi_0 = \phi$. We will show that $\eta$ is a natural isomorphism of
functors. By \cite[Lemma~A.15]{Chermak2013}, $\eta_\T$ is a rigid
isomorphism of transporter systems, provided we make the identification of $S$
with the group of of equivalence classes $\{([\delta_S(s)],S,S) \mid s \in S\}$
via the canonical isomorphism. Let now $\alpha\colon \T \to \T'$ be any
isomorphism of transporter systems, and consider the naturality diagram:
\[
\xymatrix{
\T \ar[r]^-{\eta_\T} \ar[d]_-{\alpha} & \Theta(\Lambda(\T))
\ar[d]^-{\Theta(\Lambda(\alpha))} \\ \T' \ar[r]_-{\eta_{\T'}} &
\Theta(\Lambda(\T')).
}
\]
Fix a morphism $\phi\colon P \to Q$ in $\T$.  Then 
\[
\Theta(\Lambda(\alpha))([\phi_0], P,Q) = ([\alpha(\phi_0)], \alpha(P), \alpha(Q))
\]
while 
\[
\eta_{\T'}(\alpha(\phi)) = ([\alpha(\phi)_0], \alpha(P), \alpha(Q)).
\]
where $\alpha(\phi)_0$ is the unique morphism from $\alpha(P)$ to
$Q_1 := \pi'(\alpha(\phi))(\alpha(P))$ such that $\alpha(\phi) =
\delta_{Q_1, \alpha(Q)}(1) \circ \alpha(\phi)_0$. Note also that $\alpha(\phi)
= \delta_{\alpha(Q_0), \alpha(Q)}(1) \circ \alpha(\phi_0)$ as $\alpha$
sends inclusions to inclusions. Thus, to show that $\eta$ is natural, it
suffices by uniqueness of restrictions, Lemma~\ref{L:resext}(b), to show that
$Q_1 = \alpha(Q_0)$.  To this end, let $\beta$ be the isomorphism from $S$ to
$S'$ associated with $\alpha$ in Proposition~\ref{P:aut-trans}.  By
Proposition~\ref{P:aut-trans}, $\alpha(P) = \beta(P)$ for each $P \in \Delta$,
and we have 
\[ 
\pi'(\alpha(\phi))(\alpha(P)) = c_\beta(\pi(\phi))(\beta(P)) =
\beta(\pi(\phi)(P)) = \alpha(\pi(\phi)(P)), 
\] 
as required. This completes the proof that $\eta$ is a natural isomorphism. 

Next, given a locality $(\L,\Delta,S)$ define $\zeta_\L \colon \L \to
(\Lambda\circ \Theta)(\L)$ by 
\[
\zeta_\L(f) = [(f,S_f,{ }^f\!S_f)].
\]
We will show that $\zeta = (\zeta_\L)\colon \id_{\mathsf{L}} \to \Lambda \circ
\Theta$ is a natural isomorphism. Let $(f_n,\dots,f_1) \in \bD(\L)$, and set $f
= \Pi(f_n,\dots,f_1)$. By Definition~\ref{D:locality}(L1a), there are objects
$P_0, \dots, P_{n} \in \Delta$ such that $P_{i-1} \leq S_{f_i}$ and ${
}^{f_i}\!P_{i-1} = P_{i}$ for $i = 1,\dots,n$. Then $[(f_i,S_{f_i},{
}^{f_i}\!S_{f_i})] = [(f_{i},P_{i-1},P_i)]$ by definition of the equivalence
class $[-]$, and this implies that $\zeta_\L^*(f_n,\dots,f_1) :=
(\zeta_\L(f_n),\dots,\zeta_\L(f_1)) \in \bD(\Lambda(\Theta(\L)))$.  By
definition of the product in $\Lambda(\Theta(\L))$, we have
\[
\Pi(\zeta^*_\L(f_n,\dots,f_1)) = [(\Pi(f_n,\dots,f_1),P_0,P_n)] = [(f,P_0,P_n)]
= [(f,S_f,{ }^f\!S_{f})] = \zeta_\L(\Pi(f_n,\dots,f_1)),
\]
so $\zeta_\L$ is a partial group homomorphism.

There is an extension of Lemma~3.6 of \cite{Chermak2013} in which $S$ and $S'$
(and $\Delta$ and $\Delta'$) need not be equal, and for which Chermak's proof
remains valid. This will be used to show that $\zeta_\L$ is an isomorphism of
localities.  The typical element of $\Lambda(\Theta(\L))$ has the form
$[(f,P,Q)]$ for $f \in \L$, $P \leq S_f$, and $Q \geq { }^f\!P$. It is the
image of $f$ under $\zeta_\L$, since $\zeta_\L(f) = [(f,S_f,{ }^f\!S_f)] =
[(f,P,Q)]$ by the commutative diagram
\[
\xymatrix@C+2pc{
S_f \ar[r]^{(f,S_f,{ }^f\!S_f)} & { }^f\!S_f \\ 
P  \ar[r]_{(f,P,Q)} \ar[u]^{(1,P,S_f)} & Q \ar[u]_{(1,Q, { }^f\!S_f)}
}
\]
in $\Theta(\L)$, so $\zeta_\L$ is surjective. 

Set $S' = \{[(s,S,S)]\mid s \in S\} \leq \Lambda(\Theta(\L))$, and fix $s \in
S$ and $f \in \L$. Then $(f,s,f^{-1}) \in \bD(\L)$ via $X \in \Delta$ if and
only if 
\[
([(f, { }^{sf^{-1}\!X, { }^{fsf^{-1}}\!X})], [(s,{ }^{f^{-1}}\!X, { }^{sf^{-1}}\!X)], [(f^{-1},X,{ }^{f^{-1}}\!X)]) \in \bD(\Lambda(\Theta(\L)))
\]
by definition of the domain of the locality built out of the transporter system
$\Theta(\L)$.  Moreover, in this case, $fsf^{-1} \in S$ via $X
\in \Delta$ if and only if 
\[
[(fsf^{-1},X,{ }^{fsf^{-1}}\!X)] =
[(f, { }^{sf^{-1}}\!X, { }^{fsf^{-1}}\!X) \circ (s, { }^{f^{-1}}\!X, {
}^{sf^{-1}}\!X) \circ (f^{-1},X,{ }^{f^{-1}}\!X)] \in S'
\]
This shows that $\zeta_\L(S_f) = S'_{\zeta_\L(f)}$. 

Let $h \in \ker(\zeta_\L)$. Then $[(h,S_h,{ }^h\!S_h)] =
1_{\Lambda(\Theta(\L))} = [(1,S,S)]$. This means $(h,S_h,{ }^h\!S_h)$ is a
restriction of $(1,S,S)$, that is $(1,S_h,S) = (h,S_h,S)$, and hence $h = 1$.
This completes the check of the hypotheses of the extension of
\cite[Lemma~3.6]{Chermak2013}, and so $\zeta_\L$ is an isomorphism by that
lemma. Moreover, $\zeta_\L$ is a rigid isomorphism of localities, provided we
make the identification of $S$ with the group of equivalence classes
$\{[(s,S,S)]\mid s \in S\}$ via the canonical isomorphism.

Finally, it remains to verify naturality of $\zeta$. Given another locality
$(\L',\Delta',S')$ and isomorphism $\gamma\colon \L \to \L'$ mapping $S$
onto $S'$, we have for each $f \in \L$ that 
\[
\Lambda(\Theta(\gamma))(\zeta_\L(f)) = [(\gamma(f), \gamma(S_f),\gamma({ }^f\!S_f))]'
\]
while 
\[
\zeta_\L(\gamma(f)) = [(\gamma(f), S_{\gamma(f)}, { }^{\gamma(f)}\!S_{\gamma(f)})]'
\]
As $\gamma$ is an isomorphism mapping $S$ onto $S'$, $\gamma^*(\bD_\L(f)) =
\bD_{\L'}(\gamma(f))$, and hence $\gamma(S_f) = S_{\gamma(f)}$. Also, $\gamma({
}^f\!P) = { }^{\gamma(f)}\!\gamma(P)$ for each $P \in \Delta$ and $f \in \L$.
This establishes naturality and completes the proof of the theorem.
\end{proof}

\begin{corollary}\label{C:equiv-aut}
Fix a transporter system $(\T,\pi,\delta)$ and let $\L_\Delta(\T)$ be the
associated locality. Then the map 
\[
\Phi \colon \Aut(\T) \longrightarrow \Aut(\L_{\Delta}(\T))
\]
given by sending an automorphism $\alpha \in \Aut(\T)$ to the map
$\L_{\Delta}(\T) \to \L_\Delta(\T)$ which sends a class $[\phi]$ to
$[\alpha(\phi)]$, for each $\phi \in \Iso(\T)$, is an isomorphism of groups.
Moreover, $\Phi$ maps $\Aut_{0}(\T)$ onto $\Aut_0(\L_{\Delta}(T))$. 
\end{corollary}
\begin{proof}
This follows directly from Theorem~\ref{T:local-equiv-trans}.
\end{proof}

\begin{remark}
The obstruction theory for the existence and uniqueness of centric linking
systems ``up to isomorphism'' as given by Broto, Levi, and Oliver
\cite[Theorem~3.1]{BrotoLeviOliver2003}, see also
\cite[III.5.11]{AschbacherKessarOliver2011}, holds of course with respect to
the notion of isomorphism of centric linking systems used there. By
Proposition~\ref{P:aut-trans} and Corollary~\ref{C:equiv-aut}, this definition
coincides with the notion of ``rigid isomorphism'' of the associated
localities. Thus, Theorem~3.4 of \cite{Oliver2013} and Theorem~1.1 of
\cite{GlaubermanLynd2016} imply that any two centric linking localities (i.e.,
$\Delta$-linking systems with $\Delta = \F^c$ in the terminology of
\cite[p.49]{Chermak2013}) associated to a given saturated fusion system are
rigidly isomorphic in the sense of \cite{Chermak2013}.
\end{remark}

\subsection{Linking systems and linking localities}

Theorems~\ref{T:main-loc} and \ref{T:main-trans} do not hold for arbitrary
localities and transporter systems, as can be seen by considering an
appropriate finite group $G$ of the form $O_{p'}(G) \times H$, with $O_{p'}(G)$
supporting an automorphism of order $p^2$, and forming a locality as in the
standard Example~\ref{Ex:locality}. 

\begin{definition}\label{D:linkingsystem}
A finite group $N$ is of \emph{characteristic $p$} if $C_N(O_p(N)) \leq
O_p(N)$. A \emph{linking locality} is a locality $(\L,\Delta,S)$ such that
$\F_S(\L)^{cr}\subseteq \Delta$ and $N_\L(P)$ is of characteristic $p$ for each
$P \in \Delta$. A \emph{linking system} is a transporter system
$(\T,\delta,\pi)$ associated with a fusion system $\F$ having object set
$\Delta$ such that $\F^{cr} \subseteq \Delta$ and $\Aut_\T(P)$ is of
characteristic $p$ for each $P \in \Delta$.
\end{definition}

The assumption that $\L$ is a linking locality (in
Theorem~\ref{T:main-loc}) or a linking system (in Theorem~\ref{T:main-trans})
is necessary when applying \cite[Lemma~8.2]{GlaubermanLynd2016}, which
says that a rigid automorphism of a finite group of characteristic $p$ is
conjugation by an element of the center of a Sylow $p$-subgroup.

The definition of linking system appearing in
Definition~\ref{D:linkingsystem} was given by Henke \cite{Henke2019}. It is
more general than the usual definition in
\cite[Definition~III.4.1]{AschbacherKessarOliver2011}, which forces each object
to be $\F$-\emph{quasicentric}. In Henke's definition, the objects are forced
merely to be a subset of the larger collection of $\F$-subcentric subgroups of
$S$, namely the subgroups $P$ of $S$ with the property that
$O_{p}(N_\F(Q))$ is $\F$-centric for each fully $\F$-normalized conjugate $Q$
of $P$.  The term ``linking locality'' also appears first in \cite{Henke2019}
and refers to the same thing as a ``proper locality'' in \cite{ChermakFL2}.  By
\cite[Proposition~1]{Henke2019}, the equivalence between localities and
transporter systems given in Theorem~\ref{T:local-equiv-trans} restricts to an
equivalence between linking localities and linking systems. 

Examples of linking localities include localities of finite groups of Lie type
in characteristic $p$, where, by the Borel-Tits theorem, one may take $\Delta$
to be the set of nonidentity subgroups of a Sylow subgroup. On the other hand,
every finite group $G$ gives rise to a linking locality on the set $\Delta$ of
$\F_S(G)$-subcentric subgroups of a Sylow subgroup $S$, the main theorem of
\cite{Henke2019}.

\section{Rigid outer automorphisms of centric linking
systems}\label{S:rigid-cent}

In this section, we prove Theorems~\ref{T:main-loc} and \ref{T:main-trans} in
the case $\Delta = \F^c$, and we prove Theorem~\ref{T:main-cohom}.  Throughout,
we fix a saturated fusion system $\F$ over the finite $p$-group $S$ and a
linking locality $(\L,\Delta,S)$ on $\F$.

A version of the Alperin-Goldschmidt fusion theorem for linking localities was
proved by Chermak and is needed in the proof of Theorem~\ref{T:main-loc}.  We
state a special case of it in a flexible form. 
\begin{proposition}\label{P:aft}
Let $\C$ be any conjugation family for $\F$ and let $g \in \L$.  Then there are
$Q_i \in \C \cap \Delta$ and elements $g_i \in N_\L(Q_i)$ such that $g =
g_n\cdots g_1$.
\end{proposition}
\begin{proof}
Recall, by definition of a linking locality (proper locality), that $\F^{cr}
\subseteq \Delta$. Further, the collection $\mathbf{A}(\F)$ defined in
\cite[Notation~3.3]{ChermakFL3} is a subset of $\F^{cr}$ and coincides with the
collection of $\F$-essential subgroups
\cite[Definition~I.3.2]{AschbacherKessarOliver2011}. So the assertion is a
special case of \cite[Theorem~3.5]{ChermakFL3}, given that the collection
of $\F$-essential subgroups is contained in any conjugation family, cf.
\cite[Proposition~I.3.3(b)]{AschbacherKessarOliver2011}. 
\end{proof}

Proposition~\ref{P:aft} has the immediate consequence that an automorphism
which is the identity on $N_\L(Q)$ for each $Q \in \C \cap \Delta$ is the
identity automorphism of $\L$.  We take the opportunity to prove below a more
general statement which generalizes Lemma~5.4 of \cite{GlaubermanLynd2016} to
the setting of linking localities.  We refer to
\cite[Definition~7.14]{Craven2011} for the definition of a positive
characteristic $p$-functor $W$, which we call a conjugacy functor for short.
There is a mistake in the proof of \cite[Lemma~5.4]{GlaubermanLynd2016}, in
which $W(Q)$ is claimed to be well-placed, given that $Q$ is. This seems
unlikely to be true. It is true that $W(Q)$ is conjugate to a well-placed
subgroup, and we give a correct argument in the proof of Lemma~\ref{L:5.4}.

\begin{lemma}
\label{L:5.4}
Let $\tau$ be an automorphism of $\L$. Fix a conjugacy functor $W$ for $\F$,
let $\C$ be the associated conjugation family consisting of those subgroups of
$S$ which are well-placed with respect to $W$, and set
\[
\W = \{Q \in \C \cap \Delta \mid W(Q) = Q\}.
\]
Assume that $W(Q) \in \Delta$ and $W(W(Q)) = W(Q)$ whenever $Q \in \Delta$.  If
$\tau$ is the identity on $N_\L(Q)$ for each $Q \in \W$, then $\tau$ is the
identity automorphism of $\L$.  
\end{lemma}
\begin{proof}
Assume first that $W$ is the identity functor. Then $\W = \C \cap \Delta$. Let
$\tau \in \Aut(\L)$, and assume that $\tau$ is the identity on $N_\L(Q)$ for
all $Q \in \W = \C \cap \Delta$. For $g \in \L$, there are $Q_i \in \C \cap
\Delta$ and $g_i \in N_\L(Q_i)$ such that $g = g_n \cdots g_1$ by
Proposition~\ref{P:aft}. Then $\tau(g) = \tau(g_n) \cdots \tau(g_1) = g_n
\cdots g_1 = g$ by assumption.  Thus, $\tau$ is the identity automorphism. 

Next, we prove the result for general $W$ satisfying the hypotheses. By
the previous case with the identity functor in place of $W$, it suffices
to show that $\tau$ is the identity on $N_\L(Q)$ for each $Q \in \C \cap
\Delta$. Proceed by induction on the index of $Q$ in $S$.  Assume first that $Q
= S$.  Since $S \in \C$ (it is contained in every conjugation family), $W(Q) =
W(S) \in \C \cap \Delta$ by assumption on $W$.  Hence, as $\tau|_{N_{\L}(W(S))}
= \id_{N_{\L}(W(S))}$ and $N_\L(S) \leq N_\L(W(S))$, $\tau$ is the identity on
$N_{\L}(Q)$. Fix now $Q < S$ and assume that $\tau$ is the identity on
$N_\L(R)$ for all $R \in \Delta$ with $|R| > |Q|$. Let $g \in \L$ with ${
}^g\!N_S(W(Q)) \leq S$ and ${ }^gW(Q)$ well-placed by
\cite[Lemma~7.23]{Craven2011}.  We claim that $\tau$ fixes $g$. Write $g =
g_n\cdots g_1$ for subgroups $R_i \in \C \cap \Delta$ and $g_i \in N_{\L}(R_i)$
with $R_i \geq { }^{g_i \cdots g_1}\!N_S(W(Q))$.  So $|R_i| \geq |N_S(W(Q))|
\geq |N_S(Q)| > |Q|$. The claim now follows from the inductive hypothesis.  As
${ }^gW(Q)$ is well-placed and $\Delta$ is closed under $\L$-conjugation, we
have ${ }^gW(Q) \in \C \cap \Delta$. Now $N_{\L}({ }^gQ) \leq N_\L({ }^gW(Q))$
by the axioms for a conjugacy functor. Since $\tau$ is the identity on
$N_{\L}({ }^gW(Q))$ by hypothesis, we see that $\tau$ is the identity on
$N_{\L}({ }^gQ)$. Finally, since $\tau(g) = g$, $\tau$ is the identity on
$N_\L(Q)$, as desired.  
\end{proof}

\begin{proof}[Proof of Theorem~\ref{T:main-loc} in the case $\Delta = \F^c$]
Recall that $k(p) = 1$ if $p$ is odd, and $k(p) = 2$ if $p = 2$. Fix $\tau \in
\Aut_0(\L)$. For any finite $p$-group $P$, we take the abelian version of the
Thompson subgroup $J(P)$, namely, $J(P)$ is the subgroup generated by the
abelian subgroups of $P$ of order $d(P)$, where $d(P)$ is the maximum of the
orders of the abelian subgroups of $P$. 

We proceed in several steps to complete the proof. The main part of the proof
consists in showing that if the automorphism $\tau$ is the identity on
$N_\L(J(S))$, then $\tau^{k(p)} = \id_\L$. This is carried out in Steps
2-6.

\smallskip
\noindent 
\textit{Step 1.} We first arrange that $\tau$ restricts to the identity
automorphism of $N_\L(J(S))$. The restriction $\tau$ to $N_\L(J(S))$ is an
automorphism of $N_\L(J(S))$ which is identity on $S \leq N_\L(J(S))$.  Since
$\L$ is a linking locality and $J(S) \in \Delta = \F^c$, the normalizer
$N_\L(J(S))$ is of characteristic $p$. Thus, by
\cite[Lemma~8.2]{GlaubermanLynd2016}, we may fix $z \in Z(S)$ such that $\tau$
is conjugation by $z$ on $N_{\L}(J(S))$.  Then upon replacing $\tau$ by
$c_z^{-1}\tau$, where $c_z \colon \L \to \L$ denotes the rigid inner
automorphism which is (everywhere-defined) conjugation by $z$, we complete the
proof of Step 1.

Consider the following ordering on $\F^c$:
\[
Q <_J P \quad \iff \quad d(Q) < d(P) \quad \text{ or } \quad d(Q) = d(P) \text{
and } |J(Q)| < |J(P)|.
\]
We claim that $\tau^{k(p)}$ is the identity on $\L$.  Assume the contrary,
and, using Lemma~\ref{L:5.4} with $W$ the identity functor, choose $Q$ maximal
under $<_J$ with the property that $N_\L(Q)$ is not fixed by
$\tau^{k(p)}$. 

\smallskip
\noindent
\textit{Step 2.} We show that $Q$ may be taken to be well-placed with respect
to $J$.  Let $\C$ be the collection of subgroups of $S$ which are well-placed
with respect to the Thompson subgroup functor $J$.  Then $\C$ forms a
conjugation family for $\F$ by \cite[Corollary~7.26]{Craven2011}. Let $g \in
N_{\L}(Q)$ not fixed by $\tau^{k(p)}$.  By Proposition~\ref{P:aft}, we may
write $g$ as a product of elements $g_i \in N_\L(R_i)$ with $R_i \in \C \cap
\Delta$, and where $Q = Q_0 = Q_n$, $Q_i = {}^{g_i}Q_{i-1}$, and $R_i \geq
\gen{Q_{i-1},Q_i}$ for each $i$. Since $g$ is not fixed by $\tau^{k(p)}$,
some $g_i$ is not fixed by $\tau^{k(p)}$. Now as $Q$ is isomorphic to a
subgroup of $R_i$, we see that $d(Q) \leq d(R_i)$. Therefore, equality holds by
maximality of $Q$ under $<_J$. Then $|J(Q)| \leq |J(R_i)|$, so again equality
holds by maximality of $Q$. Hence, upon replacing $Q$ by $R_i$, we may assume
that $Q \in \C$. 

\smallskip
\noindent
\textit{Step 3.} Set $H = N_\L(Q)$ and $T = N_S(Q)$. We next show that $J(Q) =
J(QJ(T))$.  As $Q \in \Delta$, $H$ is of characteristic $p$. By
\cite[Lemma~8.2]{GlaubermanLynd2016}, we may fix $z \in Z(T)$ such that $\tau$
is conjugation by $z$ on $H$. Then $\tau^2$ is conjugation by $z^2$ on $H$.
Since $\tau^{k(p)}$ is not the identity on $H$, we have that
$z^{k(p)}$ is not centralized by $H$.  Applying
\cite[Theorem~A]{Glauberman1968wc}, we conclude that $z^{k(p)}$ is
not centralized by $N_H(J(T))$. Now $N_H(J(T)) \leq N_H(QJ(T))$ since $H =
N_H(Q)$, so that $\tau^{k(p)}$ is not the identity on $N_\L(QJ(T))$.  As $QJ(T)
\in \F^c$ and $d(Q) \leq d(QJ(T))$, we have equality by maximality of $Q$ under
$<_J$. Then $J(Q) \leq J(QJ(T))$, and so 
\begin{eqnarray}
\label{E:JQ=JQJT}
J(Q) = J(QJ(T)), 
\end{eqnarray}
again by maximality of $Q$ under $<_J$. 

\smallskip
\noindent
\textit{Step 4.} Here we show $J(T) = J(Q)$.  As $d(Q) \leq d(T) = d(J(T)) \leq
d(QJ(T))$, we have equality by Step 3. Thus, $d(Q) = d(T)$ and $Q \leq T$ yield
that $J(Q) \leq J(T) \leq J(QJ(T))$, and again we have equality by choice of
$Q$. This completes the proof of Step 4. 

\smallskip
\noindent
\textit{Step 5.} We next show that $J(Q)$ is $\F$-centric.  Suppose on the
contrary that $J(Q)$ is not $\F$-centric. By Step 2, $Q$ is well-placed. By
definition of well-placed, $J(T)$ is fully $\F$-normalized. Hence, $J(Q)$ is
fully $\F$-normalized by Step 4.  Since $J(Q)$ is fully $\F$-normalized and not
$\F$-centric, we have $C_S(J(Q)) \nleq J(Q)$.  Note that $C_S(J(Q)) \nleq Q$
since $J(Q)$ does contain its centralizer in $Q$.  Hence, $QC_S(J(Q)) > Q$, so
with $R := N_{QC_S(J(Q))}(Q)$, we have 
\[
R > Q.
\]
On the other hand, Step 4 shows that
\[
R = QN_{C_S(J(Q))}(Q) = QC_T(J(Q)) = QC_T(J(T)) = QZ(J(T)) = QZ(J(Q)) = Q,
\]
a contradiction.

\smallskip
\noindent
\textit{Step 6.} Lastly, we obtain a contradiction. Among all well-placed,
$\F$-centric subgroups maximal under $<_J$ whose normalizer in $\L$ is not
centralized by $\tau^{k(p)}$, choose $Q$ of minimum order. By Step 4 and the
definition of well-placed, $J(Q) = J(T)$ is well-placed. By Step 5, $J(Q)$ is
centric.  Note $\tau^{k(p)}$ is not the identity on $N_H(J(Q)) = H$ by choice of
$Q$.  Since $d(Q) = d(J(Q))$ and $J(J(Q)) = J(Q)$, we have that $Q = J(Q)$ by
minimality of $|Q|$.  Therefore, by Step 4,
\[
J(Q) = J(T) = J(N_S(Q)) = J(N_S(J(Q))).
\]
It now follows that $Q = J(Q) = J(S)$ by
\cite[Lemma~8.5(b)]{GlaubermanLynd2016}. Since $N_\L(J(S))$ is centralized by
$\tau$ by Step 1, this is a contradiction.

\smallskip
\noindent
\textit{Step 7.} We prove the splitting condition. Since Steps 1-6 show
that $\Out_0(\L) = 1$ if $p$ is odd, splitting is trivial in that case. So take
$p = 2$. Let $E$ be the subgroup of $\Aut_0(\L)$ consisting of those
automorphisms which restrict to the identity on $N_\L(J(S))$. Step 1 shows that
$E$ maps surjectively onto $\Out_0(\L)$ via the quotient map $\Aut_0(\L) \to
\Out_0(\L)$, while Steps 1-6 show that $E$ is a vector space over
$\mathbb{F}_{2}$. There is therefore a subgroup $E_0$ which is a complement to
$C_{\Aut_{Z(S)}(\L)}(N_\L(J(S)))$ in $E$ and which maps
isomorphically onto $\Out_0(\L)$. This proves the assertion.
\end{proof}

\begin{proof}[Proof of Theorem~\ref{T:main-trans} when $\L$ is a centric linking system]
This follows directly from Theorem~\ref{T:main-loc} in the centric linking
locality case, given Theorem~\ref{T:local-equiv-trans}.
\end{proof}

\begin{remark}\label{R:stronger}
The method of proof of Theorems~\ref{T:main-loc} and \ref{T:main-trans} in case
$\Delta = \F^c$ shows the slightly stronger conclusion: if $\tau$ is an
automorphism of a centric linking locality (centric linking system) which is
the identity on $N_\L(J(S))$ ($\Aut_{\L}(J(S))$, then $\tau^{k(p)} = \id_\L$. 
\end{remark}

We next want to prove Theorem~\ref{T:main-cohom}, but first recall
certain definitions from \cite[Section III.5]{AschbacherKessarOliver2011}.
Let $\O(\F^c)$ be the category with objects the
$\F$-centric subgroups, and with morphism sets
\[
\Mor_{\O(\F^c)}(P,Q) = \Inn(Q)\backslash\Hom_{\F}(P,Q),
\]
the set of orbits of $\Inn(Q)$ in its left action by composition. The
center functor 
\[
\Z_\F\colon \O(\F^c) \to \Ab
\]
is the functor which sends a subgroup $P$ to its center $Z(P)$, and sends a
morphism $[\phi]\colon P \to Q$ to the composite $Z(Q) \hookrightarrow
Z(\phi(P)) \xra{\phi^{-1}|_{Z(\phi(P))}} Z(P)$ induced by the
restriction of $\phi^{-1}\colon \phi(P) \to P$ to $Z(\phi(P))$.

We refer to Section~III.5.1 of \cite{AschbacherKessarOliver2011} for a
description of the bar resolution for functor cohomology and write $d$ for the
coboundary map. Recall that a $0$-cochain for $\Z_\F$ sends an object $P$ of
$\O(\F^c)$ to an element in $Z(P)$. A $1$-cochain sends a morphism $P
\xra{[\phi]} Q$ in the orbit category to an element in $Z(P)$. A $1$-cochain
for $\Z_\F$ is said to be \emph{inclusion-normalized} if it sends the class of
each inclusion $\iota_{P}^Q$ to $1 \in Z(P)$. Write $\hat{Z}^1(\O(\F^c),
\Z_\F)$ for the group of inclusion-normalized $1$-cocycles, and write
$\hat{B}^1(\O(\F^c), \Z_\F) \subseteq \hat{Z}^1(\O(\F^c), \Z_\F)$ for the group
of inclusion-normalized $1$-coboundaries.  

By the proof of \cite[III.5.12]{AschbacherKessarOliver2011}, there is a group
homomorphism
\[
\tilde{\lambda}\colon \hat{Z}^1(\O(\F^c), \Z_\F) \to \Aut(\L)
\]
given by sending a $1$-cocycle $t$ to the automorphism of $\L$ which is the
identity on objects, and which sends a morphism $\phi \colon P \to Q$ in $\L$
to $\phi \circ \delta_P(t([\phi]))$.  Next, consider the group homomorphisms
\[
\mathrm{cnst} \colon Z(S) \to C^0(\O(\F^c), \Z_\F) \quad \text{ and } \quad \conj \colon Z(S) \to \Aut_0(\L),
\]
where $\mathrm{cnst}$ sends an element $z \in Z(S)$ to the constant $0$-cochain
$u_z$ with value $z$ on each centric subgroup, and $\conj$ sends an element $z$
to the conjugation automorphism $c_{\delta_S(z)} \in \Aut_0(\L)$. 

\begin{lemma}
\label{L:iso-ses}
There is an isomorphism of short exact sequences
\begin{eqnarray}
\label{E:isom-exact-seq}
\xymatrix{
1  \ar[r] & \hat{B}^1(\O(\F^c),\Z_\F) \ar[r] \ar[d]_{du \mapsto u(S)Z(\F)} & \widehat{Z}^1(\O(\F^c), \Z_\F) \ar[d]_{\tilde{\lambda}} \ar[r] & \lim{\!}^1\Z_\F \ar[r] \ar[d]_{\lambda} & 1 \\
  1 \ar[r] & Z(S)/Z(\F) \ar[r]^{\mathrm{conj}} &  \Aut_0(\L) \ar[r] &  \Out_0(\L) \ar[r] & 1. \\
}
\end{eqnarray}
\end{lemma}
\begin{proof}
This is essentially contained in the proof of
\cite[Proposition~III.5.12]{AschbacherKessarOliver2011}.  There the groups
$\Aut(\L)$ and $\Out(\L)$ are denoted $\Aut_{\typ}^I(\L)$ and
$\Out_{\typ}(\L)$. The commutative diagram displayed on
\cite[p.186]{AschbacherKessarOliver2011} is shown to have exact rows and
columns. Thus, $\tilde{\lambda}\colon \hat{Z}^1(\O(\F^c), \Z_{\F}) \to
\Aut(\L)$ is injective with image $\ker(\tilde{\mu}) = \Aut_0(\L)$.
Also, $\tilde{\lambda}$ induces an injective homomorphism $\lambda \colon
{\lim}^1\Z_\F \to \Out(\L)$ with image $\ker(\mu) = \Out_0(\L)$, and so
$\tilde{\lambda}$ and $\lambda$ are isomorphisms after restricting to these
codomains. Thus, the commutativity of this diagram also gives that the right
square in \eqref{E:isom-exact-seq} commutes.

Second, from the proof of \cite[III.5.12]{AschbacherKessarOliver2011}, the
composite $d \circ \mathrm{cnst}$ has image $\hat{B}^1(\O(\F^c), \Z_\F)$,
where, for each $z \in Z(S)$, the image $du_z$ of $u_z$ under the coboundary
map is inclusion-normalized, and $\tilde{\lambda}(du_z)$ is conjugation by
$\delta_S(z)$ on $\L$.  The composite $\hat{B}^1(\O(\F^c),\Z_\F)
\hookrightarrow \hat{Z}^1(\O(\F^c), \Z_\F) \xra{\tilde{\lambda}} \Aut_0(\L)$ is
injective. Thus, the kernel of the composite $d \circ \mathrm{cnst}$ is the
same as the kernel of $\conj$. But $\ker(\conj) = Z(\F)$ by
\cite[Lemma~1.14]{AOV2012}.  Therefore, the inverse $du \mapsto u(S)Z(\F)$ of
the isomorphism $Z(S)/Z(\F) \to \hat{B}^1(\O(\F^c), \Z_\F)$ induced by $d \circ
\mathrm{cnst}$ makes the left square in \eqref{E:isom-exact-seq} commute. 
\end{proof}

\begin{proof}[Proof of Theorem~\ref{T:main-cohom}] 
By Theorem~\ref{T:main-trans} in the case $\Delta = \F^c$, the sequence $1 \to
\Aut_{Z(S)}(\L) \to \Aut_0(\L) \to \Out_0(\L) \to 1$ is split exact.  As
$\Aut_{Z(S)}(\L)$ is the image of the conjugation map $Z(S)/Z(\F) \to
\Aut_0(\L)$, it follows from Lemma~\ref{L:iso-ses} that the sequence $1 \to
\hat{B}^1(\O(\F^c), \Z_\F) \to \hat{Z}^1(\O(\F^c), \Z_\F) \to {\lim}^1\Z_\F \to
1$ is also split exact and that ${\lim}^1 \Z_\F \cong \Out_0(\L)$ is
elementary abelian. 
\end{proof}

\section{Extending to larger object sets}\label{S:descent}

In this section, we observe via Chermak descent
\cite[Theorem~5.15]{Chermak2013} that the group of rigid automorphisms does not
change when a centric linking locality is expanded to a larger object set.
Recall from \cite{Henke2019} that a subgroup $P$ of $S$ is said to be
$\F$-\emph{subcentric} if for each fully $\F$-normalized $\F$-conjugate $Q$ of
$P$, the subgroup $O_{p}(N_\F(Q))$ is $\F$-centric. The set of $\F$-subcentric
subgroups is denoted $\F^s$.

\begin{proposition}\label{P:resiso}
Let $\L^+$ be a linking locality with object set $\Delta^+$ and fusion system
$\F$ over a $p$-group $S$. Let $\Delta \subseteq \Delta^+$ be a subset
which contains $\F^{cr}$ and is closed under $\F$-conjugacy and passing to
overgroups. Assume that $\L^+|_{\Delta} = \L$. Then restriction induces an
isomorphism $\Aut_{0}(\L^+) \to \Aut_0(\L)$ which restricts to an isomorphism
$\Aut_{Z(S)}(\L^+) \to \Aut_{Z(S)}(\L)$. 
\end{proposition}
\begin{proof}
This follows from Corollary~5.16 of \cite{Chermak2013}, applied in the same way
as in \cite[Theorem~7.2]{Henke2019}. The proof is by induction on
$|\Delta^+-\Delta|$. If $\Delta^+ = \Delta$, then $\L^+ = \L$ and there is
nothing to prove. Let $T \in \Delta^+-\Delta$ be maximal under inclusion. We
claim that Hypothesis 5.3 of \cite{Chermak2013} holds.  Since $\Delta$ and
$\Delta^+$ are $\F$-invariant and closed under passing to overgroups, we can
replace $T$ by an $\F$-conjugate if necessary and assume that $T$ is fully
$\F$-normalized. By induction, we may also assume that $\Delta^+ = \Delta \cup
T^\F$. 

Let $\hat{T} = O_p(N_\F(T))$. Then $T \leq \hat{T}$, and we claim the
inclusion is proper. Assume otherwise. As an object of a linking locality, $T$
is $\F$-subcentric by \cite[Proposition~1(b)]{Henke2019}. So by
\cite[Proposition~3.18]{Henke2019}, it follows that $T \in \F^{cr}$. But then
$T \in \Delta$, which contradicts the choice of $T$. Thus, $T < \hat{T}$, so
$\hat{T} \in \Delta$ by choice of $T$.

Let $M = N_{\L}(T)$, and set 
\[
\Delta_T := \{N_P(T) \mid T \leq P \in \Delta\} = \{P \in \Delta \mid T
\leq P \leq N_S(T)\},
\]
where the second equality comes from maximality of $T$ in $\Delta^+-\Delta$.
By Lemma~7.1 of \cite{Henke2019}, $M$ is a finite group which is a model for
$N_\F(T)$. In particular $T$ is normal in $M$ and $N_S(T)$ is a Sylow
$p$-subgroup of $M$. So indeed, taking the identity $\L \to \L$ as a rigid
automorphism, Hypothesis 5.3 of \cite{Chermak2013} holds. Recall the
locality $\L_{\Delta_T}(M)$ from Example~\ref{Ex:locality}, and note that
$\L_{\Delta_T}(M) = M$ in the current situation, since each normal
$p$-subgroup of the fusion system of $M$ is normal in $M$
\cite[Theorem~2.1(b)]{Henke2019}. By Corollary~5.16 of \cite{Chermak2013},
there is a unique rigid isomorphism $\L^+(\id_M) \to \L^+$ which restricts to
the identity on $\L$, where the former is constructed in
\cite[Theorem~5.14]{Chermak2013} and defined after the proof of
\cite[Theorem~5.14]{Chermak2013}. Identify $\L^+(\id_M)$ and $\L^+$ via this
isomorphism. The identity automorphism is then the unique rigid automorphism of
$\L^+$ which is the identity on $\L$.  This shows that the restriction map
$\Aut_0(\L^+) \to \Aut_0(\L)$ is injective. 

To see surjectivity of restriction, take an arbitrary rigid isomorphism
$\beta$ of $\L$. Again by \cite[Corollary 5.16]{Chermak2013}, there is a
rigid isomorphism $\beta^+\colon \L^+(\beta|_M) \to \L^+$ which restricts to
$\beta$ on $\L$. Taking now $\L^+(\beta_M)$ in the role of $\L^+$, we
see that there is also a rigid isomorphism $\id^+ \colon \L^+ = \L^+(\id_M)
\to \L^+(\beta_M)$ which is the identity on $\L$. The composition $\beta^+\circ
\id^+ \in \Aut_0(\L^+)$ restricts to $\beta$ on $\L$, and this shows the
restriction map is surjective.
\end{proof}

\begin{proof}[Proof of Theorems~\ref{T:main-loc} and \ref{T:main-trans}] 
Let $(\L, \Delta, S)$ be an arbitrary linking locality. Now $\Delta \subseteq
\F^s$ by Proposition~1(b) of \cite{Henke2019}, so by Theorem~7.2 of
\cite{Henke2019}, there is a linking locality $(\L^+,\F^s,S)$ which restricts
to $\L$ on $\Delta$. As $\F^c \subseteq \F^s$, two applications of
Proposition~\ref{P:resiso} give an isomorphism of short exact sequences between
$1 \to \Aut_{Z(S)}(\L) \to \Aut_0(\L) \to \Out_0(\L) \to 1$ and $1 \to
\Aut_{Z(S)}(\L^+|_{\F^c}) \to \Aut_0(\L^+|_{\F^c}) \to \Out_0(\L^+|_{\F^c}) \to
1$. Theorem~\ref{T:main-loc} now follows from the proof in the case
$\Delta = \F^c$. Then Theorem~\ref{T:main-trans} follows from
Theorem~\ref{T:main-loc} and Theorem~\ref{T:local-equiv-trans}. 
\end{proof}

\begin{remark}
Given the results of this section, the stronger statement mentioned in
Remark~\ref{R:stronger} applies verbatim to arbitrary linking localities
(linking systems) with object set $\Delta$ containing $J(S)$. 
\end{remark}

\section{Comparing automorphisms of groups and linking systems}\label{S:grp}

One may wonder whether it is possible to recover from
Theorem~\ref{T:main-trans} the analogous theorems about groups, namely
\cite[Theorem~10]{Glauberman1968wc} for $p=2$ and \cite[Theorem~3.3]{GGLN2019}
for $p$ odd. This is possible, but the only way we know how to do it goes
through an argument similar to existing arguments for establishing the group
case anyway, so our way seems to have little additional value.  However, in the
process of trying to construct a proof, we obtained Theorem~\ref{T:kerkappa}
below, which appears to be new and of independent interest. It depends for its
proof on the $Z_p^*$-theorem \cite{Glauberman1966}, \cite[7.8.2,7.8.3]{GLS3}
that in a finite group with no normal $p'$-subgroups, any element which is
weakly closed in a Sylow $p$-subgroup is central.  

First we need to set up some notation. Let $p$ be a prime and let $G$ be a
finite group with Sylow $p$-subgroup $S$. We write $\L = \L_S^c(G)$ and $\F =
\F_S(G)$ for the centric linking system and fusion system of $G$. Thus, $\L$
has objects the $\F$-centric subgroups, or equivalently, the $p$-centric
subgroups of $G$, i.e the subgroups $P$ of $S$ with $C_G(P) = Z(P) \times
O_{p'}(C_G(P))$. Morphisms are given by 
\[
\Mor_{\L}(P,Q) = N_G(P,Q)/O_{p'}(C_G(P)). 
\]
where $N_G(P,Q) = \{g \in G \mid { }^g P \leq Q\}$ is the transporter set,
where composition is induced by multiplication in $G$, and where
$O_{p'}(C_G(P))$ acts on $N_G(P,Q)$ from the right. The structural functor
$\delta$ is the inclusion map, while $\pi$ sends a coset $gO_{p'}(C_G(P))$
to conjugation by $g$.

By Sylow's theorem, each outer automorphism of $G$ is represented by an
automorphism $\alpha \in N_{\Aut(G)}(S)$.  Such an automorphism induces an
isomorphism from $O_{p'}(C_G(P))$ to $O_{p'}(C_G(\alpha(P)))$ and a
bijection $N_G(P,Q) \to N_G(\alpha(P),\alpha(Q))$, for each pair of centric
subgroups $P$ and $Q$. It is then straightforward to check that $\alpha$
induces an automorphism of $\L$ by restriction in this way.  Let 
\[ 
\tilde{\kappa}_G \colon N_{\Aut(G)}(S) \to \Aut(\L)
\]
denote the resulting group homomorphism.  This map sends $\Aut_G(S)$ onto
$\{c_\gamma \mid \gamma \in \Aut_\L(S)\}$, and so there is an induced
homomorphism
\[
\kappa_G\colon \Out(G) \to \Out(\L). 
\]
The composition $\tilde{\mu}_G \circ \tilde{\kappa}_G\colon N_{\Aut(G)}(S) \to
\Aut(\F_S(G))$ is just restriction to $S$. Here $\tilde{\mu}_G$ is defined just
after Proposition~\ref{P:aut-trans}. 

\begin{theorem}\label{T:kerkappa}
Fix a prime $p$, a finite group $G$, and a Sylow $p$-subgroup $S$ of $G$. Let
$\L$ be the centric linking system for $G$. If $O_{p'}(G) = 1$, then
$\ker(\kappa_G)$ is a $p'$-group.
\end{theorem}
The proof uses the $Z_p^*$-theorem only in the semidirect product of $G$ by a
$p$-power automorphism. So if $p=2$ or the composition factors of $G$ are
known, then this does not depend on the CFSG.
\begin{proof} 
Assume $O_{p'}(G) = 1$. Fix $a \in N_{\Aut(G)}(S)$ with $[a] \in
\ker(\kappa_G)$, and recall that $\tilde{\mu}_G \circ \tilde{\kappa}_G$ sends
$a$ to $a|_S$. Since $\tilde{\kappa}_G$ maps $N_{\Inn(G)}(S)$ onto $\Inn(\L) =
\{c_\gamma \mid \gamma \in \Aut_{\L}(S)\}$, we may adjust $a$ by an element of
$N_{\Inn(G)}(S)$ and take $a \in C_{\Aut(G)}(S)$.  Then by choice of $a$,
$\tilde{\kappa}_G(a) \in \Inn(\L) \cap \ker(\tilde{\mu}_G) = \Aut_{Z(S)}(\L)$.
Choose $z \in Z(S)$ such that $\tilde{\kappa}_G(a) = c_{z}$.  Replacing $a$ by
$ac_{z^{-1}}$, we may take $a \in \ker(\tilde{\kappa}_G)$. Finally, replacing
$a$ by a $p'$-power, we may take $a$ of $p$-power order.

We will show that, if $[a] \neq 1$ in $\Out(G)$, then $a$ normalizes but
does not centralize $H/O_{p'}(H)$ for some $p$-local subgroup $H = N_G(Q)$ with
$Q$ $p$-centric in $G$, that is, with $Q \in \F_S(G)^c$. Thus,
$\tilde{\kappa}_G(a)$ does not centralize $\Aut_\L(Q)$, and hence
$\tilde{\kappa}_G(a) \neq 1$, contrary to our choice of $a$.

So assume $[a] \neq 1$.  Let $\hat{G} = G\gen{a}$ be the semidirect product, and
set $\hat{S} = S\gen{a}$. Then $\hat{S}$ is Sylow in $\hat{G}$, and $\gen{a}
\leq Z(\hat{S})$.  Also, $\hat{S} =  S \times \gen{a}$ and $Z(\hat{S}) =
Z(S) \times \gen{a}$, etc.  Note that if $a$ is weakly closed in $\hat{S}$
with respect to $\hat{G}$, then by the $Z^*_p$-theorem, we have $a \in
Z(\hat{G})$ since $O_{p'}(\hat{G}) = O_{p'}(G) = 1$, so that $a = 1$, contrary
to assumption. 

So $a$ is not weakly closed in $\hat{S}$ with respect to $\hat{G}$. By the
Alperin-Goldschmidt fusion theorem in $\hat{G}$, there is a
$\F_{\hat{S}}(\hat{G})$-centric radical subgroup $\hat{Q} \leq \hat{S}$ and
$\hat{h} \in N_{\hat{G}}(\hat{Q})$ such that $a \in Z(\hat{S}) \leq
Z(\hat{Q})$, and $a \neq { }^{\hat{h}}a \in Z(\hat{Q})$.  By
\cite[Proposition~A.11(c)]{LeviOliver2002}, 
\begin{eqnarray}
\label{E:Qcentrad}
\text{$Q := \hat{Q} \cap G$ is $\F_S(G)$-centric radical.} 
\end{eqnarray}

Write $\hat{h} = ha^{k}$ for some integer $k$ and some $h \in G$. Since
$a^k \in \hat{Q}$ and $Q = \hat{Q} \cap G$, we have $h \in N_G(\hat{Q})
\leq N_G(Q)$. Also, $a \neq { }^{\hat{h}}a = { }^ha$. So $[a,h] \in
\hat{S}$.  Note that $a$ normalizes $N_{G}(Q)$, so $a$ normalizes
$O_{p'}(N_G(Q))$.  If $a$ centralizes $h$ modulo $O_{p'}(N_G(Q))$, then we
would have $[a,h] \in O_{p'}(N_G(Q)) \cap \hat{S} = 1$, a contradiction. Hence,
$a$ does not centralize $N_G(Q)/O_{p'}(N_G(Q))$. Together with
\eqref{E:Qcentrad}, this completes the proof of the proposition.
\end{proof}

A saturated fusion system $\F$ over $S$ is said to be \emph{tame} if $\F =
\F_S(G)$ for some finite group $G$ with Sylow $p$-subgroup $S$ such that the
map $\kappa_G$ is split surjective. Theorem~\ref{T:kerkappa} can be used to
show that the splitting condition in the definition of tame is unnecessary. 

\begin{proposition}
Let $\F$ be a saturated fusion system over the $p$-group $S$. If $\F \cong
\F_S(G)$ for some finite group $G$ such that the map $\kappa_G$ is surjective,
then $\F$ is tame. 
\end{proposition}
\begin{proof}
Fix such a $G$, let $\bar{G} = G/O_{p'}(G)$, and identify $S$ also with its
image in $\bar{G}$. Write $\F = \F_S(G)$, $\bar{\F} = \F_S(\bar{G})$, $\L =
\L_S^c(G)$, and $\bar{\L} = \L_S^c(\bar{G})$. The canonical homomorphism $G \to
\bar{G}$ induces isomorphisms $\L \to \bar{\L}$ and $\F \to \bar{\F}$.  As
in the proof of \cite[Lemma~2.19]{AOV2012}, there is a resulting
commutative diagram
\[
\xymatrix{
\Out(G) \ar[r] \ar[d]_{\kappa_G} & \Out(\bar{G}) \ar[d]^{\kappa_{\bar{G}}}\\
\Out(\L) \ar[r]^{\cong} & \Out(\bar{\L})
}
\]
As $\kappa_{G}$ is surjective, also $\kappa_{\bar{G}}$ is surjective, so we may
replace $G$ by $\bar{G}$ and take $O_{p'}(G) = 1$. The result now follows from
Theorem~\ref{T:kerkappa} and \cite[Lemma~1.5(b)]{BrotoMollerOliver2019}.
\end{proof}

In \cite{Glauberman1966b}, the first author showed, for a core-free group $G$
with Sylow $2$-subgroup $S$, that the group $C_{\Aut(G)}(S)$ has abelian
$2$-subgroups and a normal $2$-complement. The following proposition gives
further information and a reinterpretation of that situation.

\begin{proposition}
Let $G$ be a finite group with Sylow $p$-subgroup $S$, let $\L$ be the centric
linking system for $G$, and set $A = C_{\Aut(G)}(S)/C_{\Inn(G)}(S)$. If
$O_{p'}(G) = 1$, then $A \cong O_{p'}(A) \rtimes B$ where $B = 1$ if $p$ is
odd, and where $B$ is an elementary abelian $2$-group if $p =2$.  The normal
$p$-complement $O_{p'}(A)$ is the subgroup of $N_{\Aut(G)}(S)/N_{\Inn(G)}(S)$
consisting of those classes which have a representative that restricts to the
identity on $\L$. In particular, $\kappa_G$ is injective upon restriction
to any Sylow $p$-subgroup of $\Out(G)$.
\end{proposition}
\begin{proof}
Note that $A$ is the kernel of the composite $\mu_G \circ \kappa_G$,
which is induced by restriction to $S$. By Theorem~\ref{T:main-trans}, the
kernel of $\mu_G$ is either $1$ or an elementary $2$-group in the cases $p$ odd
or $p=2$, respectively. So $\ker(\kappa_G) = O_{p'}(A)$ by
Theorem~\ref{T:kerkappa}. The last statement follows immediately.
\end{proof}

\bibliographystyle{amsalpha}{ }
\bibliography{/home/justin/work/math/research/mybib}
\end{document}